\documentclass[a4paper,reqno,11pt]{amsart}
\usepackage{graphics}
\usepackage{amssymb}
\usepackage{amstext}
\usepackage{amsmath}
\usepackage{amscd}
\usepackage{amsthm}
\usepackage{amsfonts}
\usepackage{enumerate}
\usepackage{color}
\usepackage{here} % 図や表を強制的に出力させる
\usepackage{bm} % 太字ベクトル
\usepackage{mathrsfs}
\usepackage{mathtools}
\usepackage{latexsym}
\usepackage{booktabs}
\usepackage{fancyhdr}
\usepackage{subcaption}
\usepackage{tikz}
\usepackage{tikz-cd}
\usetikzlibrary{arrows}
\usetikzlibrary{decorations.markings}
\usetikzlibrary{patterns,decorations.pathreplacing}


\makeatletter
  
\@addtoreset{equation}{section}
\makeatother

\newcommand{\calB}{\mathcal{B}}
\newcommand{\calC}{\mathcal{C}}

\newcommand{\calO}{\mathcal{O}}

\newcommand{\calM}{\mathcal{M}}

\newcommand{\ZZ}{\mathbb{Z}}
\newcommand{\QQ}{\mathbb{Q}}
\newcommand{\RR}{\mathbb{R}}

\newcommand{\kk}{\Bbbk}

\newcommand{\ab}{\mathbf{a}}

\newcommand{\eb}{\mathbf{e}}

\newcommand{\xb}{\mathbf{x}}

\newcommand{\Hom}{\operatorname{Hom}}

\def\opn#1#2{\def#1{\operatorname{#2}}} % to make operators
\opn\Cl{Cl} \opn\conv{conv} \opn\deg{deg} \opn\rank{rank} \opn\Spec{Spec} \opn\Stab{Stab}
\opn\cone{cone} \opn\End{End} \opn\Hom{Hom} \opn\mod{mod} \opn\gldim{gldim} \opn\pdim{pdim}
\opn\Block{Block} \opn\Pyr{Pyr}

\newtheorem{thm}{Theorem}[section]

\newtheorem{lem}[thm]{Lemma}
\newtheorem{prop}[thm]{Proposition}

\theoremstyle{definition}
\newtheorem{defi}[thm]{Definition}

\theoremstyle{remark}
\newtheorem{rem}[thm]{Remark}

\begin{document}

\title{Three families of toric rings arising from posets or graphs with small class groups}
\author{Akihiro Higashitani}
\author{Koji Matsushita}

\address[A. Higashitani]{Department of Pure and Applied Mathematics, Graduate School of Information Science and Technology, Osaka University, Suita, Osaka 565-0871, Japan}
\email{higashitani@ist.osaka-u.ac.jp}
\address[K. Matsushita]{Department of Pure and Applied Mathematics, Graduate School of Information Science and Technology, Osaka University, Suita, Osaka 565-0871, Japan}
\email{k-matsushita@ist.osaka-u.ac.jp}


\maketitle

%%%---abstract---%%%
\begin{abstract} 
The main objects of the present paper are 
(i) Hibi rings (toric rings arising from order polytopes of posets), 
(ii) stable set rings (toric rings arising from stable set polytopes of perfect graphs), and 
(iii) edge rings (toric rings arising from edge polytopes of graphs satisfying the odd cycle condition). 
The goal of the present paper is to analyze those three toric rings and 
to discuss their structures in the case where their class groups have small rank. 
We prove that the class groups of (i), (ii) and (iii) are torsionfree. 
More precisely, we give descriptions of their class groups. 
Moreover, we characterize the posets or graphs whose associated toric rings have rank $1$ or $2$. 
By using those characterizations, we discuss the differences of isomorphic classes of those toric rings with small class groups. 
\end{abstract}

\bigskip

\section{Introduction}

\subsection{Background}

Toric rings of lattice polytopes are of particular interest in the area of combinatorial commutative algebra. 
Especially, the following three toric rings have been well studied: 
\begin{itemize}
\item Hibi rings, which are toric rings arising from order polytopes; 
\item stable set rings, which are toric rings arising from stable set polytopes; 
\item edge rings, which are toric rings arising from edge polytopes. 
\end{itemize}
For the precise definitions of those toric rings, see Section~\ref{sec:pre}. 
%Note that those are toric rings of lattice polytopes associated with posets or graphs. 
The goal of the present paper is to understand those three toric rings from viewpoints of class groups. 
Specifically, what we would like to do is to give characterizations of Hibi rings, stable set rings and edge rings 
in the case where their class groups have small rank and to discuss the relationships among them.

\subsection*{Order polytopes and Hibi rings}
Let $\calO_\Pi$ denote the order polytope of a given finite poset $\Pi$. 
Order polytopes of posets were introduced by Stanley (\cite{S86}). 
Around that time, Hibi introduced a class of normal Cohen--Macaulay domains $\kk[\Pi]$ arising from posets $\Pi$ (\cite{H87}), 
and it turned out that $\kk[\Pi]$ is isomorphic to the toric ring of the order polytope of $\Pi$. 
Since then, the toric rings of order polytopes of posets $\Pi$ (i.e., $\kk[\Pi]$) are called Hibi rings of $\Pi$. 
A typical example of Hibi rings is Segre products of polynomial rings (see, e.g., \cite[Example 2.6]{HN}). 
Recently, algebraic properties of Hibi rings have been well studied. 
For example, class groups of Hibi rings of posets $\Pi$ are completely characterized in terms of the Hasse diagrams of $\Pi$ (\cite[Theorem]{HHN}). 
Moreover, by using that description, conic divisorial ideals of Hibi rings are also completely described (\cite[Theorem 2.4]{HN}). 
Furthermore, in \cite{N}, the Gorenstein Hibi rings whose class groups are of rank $2$ are investigated from viewpoints of non-commutative crepant resolutions. 

\subsection*{Stable set polytopes and stable set rings}
Let $\Stab_G$ denote the stable set polytope of a given finite simple graph $G$. 
Note that toric rings of stable set polytopes are called stable set rings in \cite{HS}, so we also employ this terminology. 
Stable set polytopes of graphs were introduced by Chv\'{a}tal (\cite{Ch}). Stable set polytopes behave well for perfect graphs. 
For example, the facets of stable set polytopes are completely characterized in the case of perfect graphs (\cite[Theorem 3.1]{Ch}). 
Moreover, it is known that stable set polytopes of perfect graphs are compressed, so those are normal (see, e.g., \cite{MHS}). 
It is noteworthy that stable set polytopes include a class of another kind of polytopes arising from posets, 
called chain polytopes, which were also introduced by Stanley (\cite{S86}). 
%More precisely, we see that for a given poset $\Pi$, the stable set polytope of the comparability graph $G(\Pi)$ of $\Pi$ 
%coincides with the chain polytope of $\Pi$ (see below Proposition~\ref{char:order}). 

\subsection*{Edge polytopes and edge rings}
Let $P_G$ denote the edge polytope of a given finite simple graph $G$. 
Toric rings of edge polytopes of graphs $G$ are known as edge rings of $G$. 
Edge polytopes and edge rings began to be studied by Ohsugi--Hibi (\cite{OH98}) and Simis--Vasconcelos--Villarreal (\cite{SVV}). 
It is proved in \cite{OH98, SVV} that the edge ring of a graph $G$ is normal if and only if $G$ satisfies the odd cycle condition. 
Note that class groups and conic divisorial ideals of edge rings of complete multipartite graphs are investigated in \cite{HM}. 
The toric ideals of graphs and their Gr\"obner basis have been well studied since edge rings were introduced. 
We refer the readers to e.g., \cite[Section 5]{HHO} and \cite[Chapters 10 and 11]{Villa} for the introductions to edge rings or toric ideals of graphs.

\medskip

Our goal is to study those three families of toric rings from viewpoints of their class groups. 
Namely, we discuss the relationships among those toric rings in the case where their class groups have small rank. 

\medskip

\subsection{Results}

Before comparing our three toric rings, we study their class groups in terms of the underlying posets or graphs. 
The first main result is the torsionfreeness of our toric rings: 
\begin{thm}[{See Proposition~\ref{prop:class_stab} and Theorem~\ref{thm:class_edge}}]\label{thm:torsionfree}
Class groups of stable set rings of perfect graphs and edge rings of graphs satisfying the odd cycle condition are torsionfree. 
\end{thm}
Note that the class groups of Hibi rings are already characterized in \cite{HHN} and the torsionfreeness also holds. 
Thus, all of our toric rings have torsionfree class groups.

\medskip

The second main results are the characterizations of our toric rings with their class groups $\ZZ$ or $\ZZ^2$: 
\begin{itemize}
\item We characterize the posets $\Pi$ whose Hibi rings have the class groups $\ZZ$ or $\ZZ^2$. See Proposition~\ref{char:order}. 
Remark that this characterization is essentially obtained in \cite[Example 3.1 and Lemma 3.2]{N}. 
\item We characterize the perfect graphs $G$ whose stable set rings have the class groups $\ZZ$ or $\ZZ^2$. See Theorem~\ref{thm:stab}. 
In this case, we can see that each stable set ring is isomorphic to a certain Hibi ring. 
\item We characterize the $2$-connected graphs $G$ whose edge rings have the class groups $\ZZ$ or $\ZZ^2$. 
See Theorem~\ref{thm:bip_class_group} in the case where $G$ is bipartite and Theorem~\ref{thm:nonbip_class_group} in the case where $G$ is non-bipartite. 
Similarly to stable set rings, we can see that each edge ring is isomorphic to a certain Hibi ring. 
\end{itemize}

\medskip

Let ${\bf Order}_n$, ${\bf Stab}_n$ and ${\bf Edge}_n$ denote the sets of unimodular equivalence classes 
of order polytopes, stable set polytopes and edge polytopes such that the associated toric rings have class groups of rank $n$, respectively. 
Namely, those correspond to the sets of isomorphic classes of Hibi rings, stable set rings and edge rings whose class groups have rank $n$, respectively. 
The following relationships follow from the characterizations mentioned above together with some additional examples: 
\begin{itemize}
\item ${\bf Order}_1={\bf Stab}_1={\bf Edge}_1$ (see Subsection~\ref{sec:n=1}); 
\item ${\bf Stab}_2 \cup {\bf Edge}_2 = {\bf Order}_2$ and there is no inclusion between ${\bf Stab}_2$ and ${\bf Edge}_2$ (see Subsection~\ref{sec:n=2}); 
\item there is no inclusion among ${\bf Order}_3$, ${\bf Stab}_3$ and ${\bf Edge}_3$ (see Subsection~\ref{sec:n=3}). 
\end{itemize}

\medskip

\subsection{Organization}

A brief organization of the present paper is as follows. 
In Section~\ref{sec:pre}, we recall some fundamental materials, e.g., toric rings of lattice polytopes and their properties, 
and the definitions of order polytopes, stable set polytopes and edge polytopes. 
We also recall several properties of those polytopes or the associated toric rings. 
In Section~\ref{sec:torsionfreee}, we give a description of class groups of stable set rings of perfect graphs (Proposition~\ref{prop:class_stab}) 
and that of edge rings of connected graphs satisfying the odd cycle condition (Theorem~\ref{thm:class_edge}) in terms of the underlying graphs. 
In Section~\ref{sec:small}, we focus on the case where the class groups of our three toric rings have rank $1$ or $2$. 
Under this assumption, we provide a characterization of Hibi rings (Proposition~\ref{char:order}, but this is essentially obtained in \cite{N}), 
that of stable set rings (Theorem~\ref{thm:stab}) and that of edge rings (Theorems~\ref{thm:bip_class_group} and \ref{thm:nonbip_class_group}). 
In Section~\ref{sec:relation}, we discuss the relationships among ${\bf Order}_n$, ${\bf Stab}_n$ and ${\bf Edge}_n$ in the case where $n \leq 3$.

\medskip

%%%%%%%%%%%%%%%%%%%%%%%%%%%%%%%%%%%%%%%%%%%%%%%%%%%%%%%%%%%%%%%%%%%%%%%%%%
\subsection*{Acknowledgement} 
%The authors would like to thank Yusuke Nakajima for a lot of his helpful comments on the results. 
The first named author is partially supported by JSPS Grant-in-Aid for Scientists Research (C) 20K03513. 
%%%%%%%%%%%%%%%%%%%%%%%%%%%%%%%%%%%%%%%%%%%%%%%%%%%%%%%%%%%%%%%%%%%%%%%%%%

\bigskip

%%%%%%%%%%%%%%%%%%%%%%%%%%%%%%%%%%%%%%%%%%%%%%%%%%%%%%%%%%%%%%%%%%%%%%%%%%%%%%%%%%%%%%%%%%%%%%%%%%%%%%%%%%%%%%%%%%%%%%%%%%%%%
%%%%%%%%%%%%%%%%%%%%%%%%%%%%%%%%%%%%%%%%%%%%%%%%%%%%%%%%%%%%%%%%%%%%%%%%%%%%%%%%%%%%%%%%%%%%%%%%%%%%%%%%%%%%%%%%%%%%%%%%%%%%%
%%%%%%%%%%%%%%%%%%%%%%%%%%%%%%%%%%%%%%%%%%%%%%%%%%%%%%%%%%%%%%%%%%%%%%%%%%%%%%%%%%%%%%%%%%%%%%%%%%%%%%%%%%%%%%%%%%%%%%%%%%%%%
%%%%%%%%%%%%%%%%%%%%%%%%%%%%%%%%%%%%%%%%%%%%%%%%%%%%%%%%%%%%%%%%%%%%%%%%%%%%%%%%%%%%%%%%%%%%%%%%%%%%%%%%%%%%%%%%%%%%%%%%%%%%%
\section{Preliminaries}\label{sec:pre}

The goal of this section is to prepare the required materials for the discussions of the class groups of our toric rings. 

\subsection{Toric rings and Ehrhart rings of lattice polytopes} 
Let us recall the toric rings of lattice polytopes. 
We refer the readers to e.g., \cite{BG2} or \cite{Villa}, for the introduction. 

We call $P \subset \RR^d$ a \textit{lattice polytope} if $P$ is a convex polytope whose vertices sit in $\ZZ^d$. 
Let $\kk$ be a field and let $P \subset \RR^d$ be a lattice polytope. We define the toric ring by setting 
$$\kk[P]=\kk[\xb^\ab t : \ab \in P \cap \ZZ^d],$$ 
where $\xb^\ab =x_1^{a_1}\cdots x_d^{a_d}$ for $\ab=(a_1,\ldots,a_d)\in \ZZ^d$. 
Then $\kk[P]$ is standard graded by setting $\deg(\xb^\ab t)=1$ for each $\ab \in P \cap \ZZ^d$. 
The Krull dimension of $\kk[P]$, denoted by $\dim \kk[P]$, is equal to the dimension of $P$ plus 1, i.e., $\dim \kk[P]=\dim P+1$.

\medskip

Let $P \subset \RR^d$ be a lattice polytope. 
We say that $P$ has the \textit{integer decomposition property} (\textit{IDP}, for short) 
if for any positive integer $n$ and any $\alpha \in nP \cap \ZZ^d$, there exist $\alpha_1,\ldots,\alpha_n \in P \cap \ZZ^d$ 
such that $\alpha=\alpha_1+\cdots+\alpha_n$. 
Note that $\kk[P]$ coincides with the Ehrhart ring of $P$ if $P$ has IDP. (See \cite[Section 10.4]{Villa} for Ehrhart rings.) 
In what follows, we call a lattice polytope which has IDP an \textit{IDP polytope}. 

We say that two lattice polytopes $P,P' \subset \RR^d$ are \textit{unimodularly equivalent} 
if there are a lattice vector ${\bf v} \in \ZZ^d$ and a unimodular transformation $f \in \mathrm{GL}_d(\ZZ)$ such that $P'=f(P)+{\bf v}$. 
Regarding the unimodular equivalence of IDP polytopes and the equivalence of toric rings as graded algebras, we know that 
for two IDP polytopes $P$ and $Q$, the toric rings $\kk[P]$ and $\kk[Q]$ are isomorphic as graded algebras 
if and only if $P$ and $Q$ are unimodularly equivalent (\cite[Theorem 5.22]{BG2}).

\medskip

Let $\langle \cdot , \cdot \rangle$ denote the natural inner product of $\RR^d$. For ${\bf v} \in \RR^d$ and $b \in \RR$, 
we denote by $H^{(+)}({\bf v} ; b)$ (resp. $H({\bf v} ; b)$) an affine half-space $\{{\bf u} \in \RR^d : \langle {\bf u}, {\bf v}\rangle \ge -b\}$ 
(resp. an affine hyperplane $\{{\bf u} \in \RR^d : \langle {\bf u} , {\bf v} \rangle = -b\}$). 
For each facet $F$ of a lattice polytope $P \subset \RR^d$ of dimension $d$, 
there exist a unique primitive lattice vector ${\bf n}_F \in \RR^d$ and integers $p_{F}, q_{F}$ 
with $\mathrm{gcd}(p_{F}, q_{F}) = 1$ and $q_{F}>0$ such that $P \cap H({\bf n}_{F} ; p_{F}/q_{F})=F$, 
where a vector ${\bf n} = (n_1,\ldots,n_N) \in  \ZZ^d$ is called \textit{primitive} if the greatest common divisor of $|n_i|$'s with $n_i \neq 0$ is equal to $1$. 
Let \begin{align*}%\label{eq:phi}
\Phi(P)=\{ \tilde{{\bf n}}_{F}=(q_{F}{\bf n}_{F},p_{F})\in \ZZ^d \times \ZZ : F \text{ is a facet of $P$}\}.\end{align*}

Let $\Cl(R)$ denote the class group of a toric ring $R$. 
For the computation of class groups of $\kk[P]$, the following is known: 
\begin{lem}[cf. {\cite[Corollary]{HHN}}]\label{lem class group}
Let $P$ be an IDP polytope of dimension $d$. 
Then the rank $\dim_{\QQ}\Cl(\kk[P]) \otimes_\ZZ \QQ$ of the class group $\Cl(\kk[P])$ is equal to $|\Phi(P)| - (d+1)$. 
Moreover, $\Cl(\kk[P])$ is torsionfree if there exist $d+ 1$ distinct facets $F_i\ (i= 1,\ldots,d+ 1)$ of $P$ 
such that $\det(\tilde{{\bf n}}_{F_1},\ldots,\tilde{{\bf n}}_{F_{d+1}}) = \pm 1$. 
\end{lem}

\medskip

Each supporting hyperplane in $P$ can be identified with a linear form. 
Note that the linear form which gives a hyperplane $H$ is not uniquely determined, 
but for a lattice polytope $P$ and a supporting hyperplane $H$ of $P$, 
we can define a unique linear form $\ell_H \in \QQ^d$ with the following condition: 
\begin{itemize}
\item[(i)] $\langle \ell_H, \alpha \rangle \in \ZZ$ for any $\alpha \in P \cap \ZZ^d$; \quad
(ii) $\sum_{\alpha \in P \cap \ZZ^d}\langle \ell_H,\alpha\rangle\ZZ=\ZZ$. 
\end{itemize}
Let \begin{align}\label{eq:psi}\Psi(P)=\{\ell_H : H \text{ is a supporting hyperplane of }P\}.\end{align}

Let $P \subset \RR^d$ be an IDP polytope. 
Given $\alpha \in P \cap \ZZ^d$, we define ${\bf w}_\alpha$ belonging to a free abelian group $\bigoplus_{\ell \in \Psi(P)} \ZZ \eb_\ell$ 
with its basis $\{\eb_\ell\}_{\ell \in \Psi(P)}$ as follows: 
$${\bf w}_\alpha=\sum_{\ell \in \Psi(P)}\langle \ell,\alpha \rangle \eb_\ell.$$
Let $\calM$ be the matrix whose column vectors consist of ${\bf w}_\alpha$ for $\alpha \in P \cap \ZZ^d$. 
%Note that $P_G \cap \ZZ^d$ one-to-one corresponds to $E(G)$ by definition. 
In \cite[Section 9.8]{Villa}, the class groups of normal toric rings are discussed. By using those theories, we see the following: 
\begin{prop}[{cf. \cite[Theorem 9.8.19]{Villa}}]\label{prop:class_group}
Work with the same notation as above. Assume that $\Psi(P)$ is irredundant. Then 
$$\Cl(\kk[P]) \cong \bigoplus_{\ell \in \Psi(P)} \ZZ \eb_\ell \big/ \sum_{\alpha \in P \cap \ZZ^d} \ZZ {\bf w}_\alpha.$$ 
In particular, we have 
$$\Cl(\kk[P]) \cong \ZZ^t \oplus \ZZ/d_1\ZZ \oplus \cdots \oplus \ZZ/d_s\ZZ,$$ 
where $s=\rank \calM$, $t=|\Psi(P)|-s$ and $d_1,\ldots,d_s$ are positive integers appearing in the diagonal of the Smith normal form of $\calM$. 
\end{prop}

\medskip

We also recall the notion of a lattice pyramid over a lattice polytope. 
Let $P \subset \RR^d$ be a lattice polytope. We define a new lattice polytope as follows: 
$$\Pyr(P)=\conv( P \times \{0\} \cup \{\eb_{d+1}\}) \subset \RR^{d+1}.$$
We call $\Pyr(P)$ a \textit{lattice pyramid} over $P$. We can see that $$\kk[\Pyr(P)] \cong \kk[P] \otimes_\kk \kk[x].$$ 
In particular, $\Cl(\kk[\Pyr(P)]) \cong \Cl(\kk[P])$. 

\bigskip

\subsection{Hibi rings: toric rings of order polytopes}
In this subsection, we recall what Hibi rings and order polytopes of posets are. 

Let $\Pi$ be a finite partially ordered set (poset, for short) equipped with a partial order $\prec$. 
%For $p,q \in \Pi$, we say that \textit{$p$ covers $q$} if $q \prec p$ and there is no $p' \in \Pi \setminus \{p,q\}$ with $q \prec p' \prec p$. 
For a subset $I \subset \Pi$, we say that $I$ is a \textit{poset ideal} of $\Pi$ if $p \in I$ and $q \prec p$ imply $q \in I$. 
For a subset $A \subset \Pi$, we call $A$ an \textit{antichain} of $\Pi$ if $p \not\prec q$ and $q \not\prec p$ for any $p,q \in A$ with $p \neq q$. 
Note that $\emptyset$ is regarded as a poset ideal and an antichain. 

For a poset $\Pi=\{p_1,\ldots,p_d\}$, let \begin{align*}
\calO_\Pi=\{(x_1,\ldots,x_d) \in \RR^d : \; x_i \geq x_j \text{ if } p_i \prec p_j \text{ in }\Pi, \;\; 
0 \leq x_i \leq 1 \text{ for }i=1,\ldots,d \}.
\end{align*}
A convex polytope $\calO_\Pi$ is called the \textit{order polytope} of $\Pi$. 
It is known that $\calO_\Pi$ is a $(0,1)$-polytope and the vertices of $\calO_\Pi$ one-to-one correspond to the poset ideals of $\Pi$ (\cite{S86}). 
In fact, a $(0,1)$-vector $(a_1,\ldots,a_d)$ is a vertex of $\calO_\Pi$ if and only if $\{ p_i \in \Pi : a_i =1 \}$ is a poset ideal. 
The toric ring $\kk[\calO_\Pi]$ is called the \textit{Hibi ring} of $\Pi$. 
We denote the Hibi ring of $\Pi$ by $\kk[\Pi]$ instead of $\kk[\calO_\Pi]$ for short.

\smallskip

We also recall another polytope arising from $\Pi$, which is defined as follows: 
\begin{align*}
\calC_\Pi=\{(x_1,\ldots,x_d) \in \RR^d : \;&x_i \geq 0 \text{ for }i=1,\ldots,d, \\
&x_{i_1}+\cdots+x_{i_k} \leq 1 \text{ for }p_{i_1} \prec \cdots \prec p_{i_k} \text{ in }\Pi\}.\end{align*} 
A convex polytope $\calC_\Pi$ is called the \textit{chain polytope} of $\Pi$. 
Similarly to order polytopes, it is known that $\calC_\Pi$ is a $(0,1)$-polytope and the vertices of $\calC_\Pi$ one-to-one correspond to the antichains of $\Pi$ (\cite{S86}). 
%In fact, a $(0,1)$-vector $(a_1,\ldots,a_{d-1})$ is a vertex of $\calO_\Pi$ if and only if $\{ p_i \in \Pi : a_i =1 \}$ is a poset ideal. 

\smallskip

In general, the order polytope and the chain polytope of $\Pi$ are not unimodularly equivalent, 
but the following is known: 

\begin{thm}[{\cite[Theorem 2.1]{HL16}}]\label{X}
Let $\Pi$ be a poset. Then $\calO_\Pi$ and $\calC_\Pi$ are unimodularly equivalent if and only if $\Pi$ does not contain the ``X-shape'' subposet. 
\end{thm}
Here, the ``X-shape'' poset is the poset $\{z_1,z_2,z_3,z_4,z_5\}$ equipped with the partial orders $z_1 \prec z_3 \prec z_4$ and $z_2 \prec z_3 \prec z_5$.

\bigskip

\subsection{Stable set rings: toric rings of stable set polytopes}
In this subsection, we recall stable set polytopes of graphs. 
For the fundamental materials on graph theory, consult, e.g., \cite{Die}. 

Let $G$ be a finite simple graph on the vertex set $V(G)=[d]$ with the edge set $E(G)$, where we let $[d]=\{1,\ldots,d\}$ for $d \in \ZZ_{>0}$. 
Throughout the present paper, we only treat finite simple graphs, so we simply call graphs instead of finite simple graphs. 
We say that $T \subset V(G)$ is an \textit{independent set} or a \textit{stable set} (resp. a \textit{clique}) 
if $\{v,w\} \not\in E(G)$ (resp. $\{v,w\} \in E(G)$) for any distinct vertices $v,w \in T$. 
Note that the empty set and each singleton are regarded as independent sets, and we call such independent sets \textit{trivial}. 

Given a subset $W \subset V(G)$, let $\rho(W)=\sum_{i\in W}\eb_i$, 
where $\eb_i$ denotes the $i$th unit vector of $\RR^d$ for $i \in [d]$ and we let $\rho(\emptyset)$ be the origin of $\RR^d$. 
We define a lattice polytope associated with a graph $G$ as follows: 
\begin{align*}
\Stab_G=\conv(\{\rho(W) : W \text{ is a stable set}\}). 
\end{align*}
We call $\Stab_G$ the \textit{stable set polytope} of $G$.

In what follows, we treat the stable set rings of \textit{perfect graphs}. 
The reason why we focus on perfect graphs is derived from the following: 
\begin{itemize}
\item $\Stab_G$ is compressed if and only if $G$ is perfect, thus, $\Stab_G$ is normal if $G$ is perfect. 
\item the facets of $\Stab_G$ are completely characterized when $G$ is perfect (\cite[Theorem 3.1]{Ch}). 
More concretely, the facets of $\Stab_G$ are exactly defined by the following hyperplanes: 
\begin{equation}\label{facets:stab}
\begin{split}
&H(\eb_i;0) \;\; \text{ for each } i \in [d]; \\ 
&H\left(-\sum_{j\in Q} \eb_j ; 1\right) \;\;\text{ for each maximal clique }Q. 
\end{split}
\end{equation}
\end{itemize}

We prepare some more notation on graphs. 
For a subset $W \subset V(G)$, let $G_W$ denote the induced subgraph with respect to $W$. 
For a vertex $v$, we denote by $G \setminus v$ instead of $G_{V(G) \setminus \{v\}}$. 
Similarly, for $S \subset V(G)$, we denote by $G \setminus S$ instead of $G_{V(G) \setminus S}$. 
%We say that $T \subset V(G)$ an \textit{independent set} (or \textit{stable set}) if $\{v,w\} \not\in E(G)$ for any two vertices $v,w \in T$. 
For a subgraph $G'$ of $G$ and $S \subset V(G)$, we define $G'+S$ to be the subgraph of $G$ on the vertex set $V(G')\cup S$ 
with the edge set $E(G')\cup \{\{v,w\} : v\in S, w\in V(G'), \{v,w\} \in E(G)\}$. 
Similarly, for $v\in V(G)$, we denote by $G'+v$ instead of $G'+\{v\}$. 
Given $v \in V(G)$, let $N_G(v)=\{w \in V(G) : \{v,w\} \in E(G)\}$. For $S \subset V(G)$, let $N_G(S)=\bigcup_{v \in S}N_G(v)$. 

\bigskip

\subsection{Edge rings: toric rings of edge polytopes}
In this subsection, we recall what edge rings and edge polytopes of graphs are. 

For a positive integer $d$, consider a graph $G$ on the vertex set $V(G)=[d]$ with the edge set $E(G)$. 
%Given an edge $e=\{i,j\} \in E(G)$, let $\rho(e)=\eb_i+\eb_j$. 
We define a lattice polytope associated to $G$ as follows: 
\begin{align*}
P_G=\conv(\{\rho(e) : e \in E(G)\}). 
\end{align*}
We call $P_G$ the \textit{edge polytope} of $G$. 

Moreover, we also define the edge ring of $G$, denoted by $\kk[G]$, 
as a subalgebra of the polynomial ring $\kk[t_1,\ldots,t_d]$ in $d$ variables over a field $\kk$ as follows: 
$$\kk[G] = \kk[t_it_j : \{i,j\} \in E(G)].$$ 
%This is actually a monoid $\kk$-algebra associated to the monoid $\ZZ_{\geq 0}(P_G \cap \ZZ^d)$, 
%i.e., the edge ring is the toric ring (a.k.a. the polytopal monomial subring) of the edge polytope. 
Note that the edge ring of $G$ is nothing but the toric ring of $P_G$. 
We have that $\dim P_G=d-b(G)-1$, where $b(G)$ is the number of bipartite connected components of $G$ (see \cite[Proposition 10.4.1]{Villa}). 
%if $G$ is non-bipartite (\cite[Proposition 1.3]{OH98}). 
Thus, $\dim\kk[G]=d-b(G)$. %if $G$ is non-bipartite. 

\medskip

It is known that $\kk[G]$ is normal (i.e., $P_G$ has IDP) if and only if $G$ satisfies the \textit{odd cycle condition}, i.e., 
for each pair of odd cycles $C$ and $C'$ with no common vertex, there is an edge $\{v,v'\}$ with $v \in V(C)$ and $v' \in V(C')$ 
(see \cite[Corollary 10.3.11]{Villa}).

\medskip

The following terminologies are used in \cite{OH98}: 
\begin{itemize}
\item We call a vertex $v$ of $G$ \textit{regular} (resp., \textit{ordinary}) if each connected component of $G \setminus v$ contains an odd cycle (resp., if $G \setminus v$ is connected). 
Note that a non-ordinary vertex is usually called a \textit{cut vertex} of $G$. 
\item Given an independent set $T \subset V(G)$, let $B(T)$ denote the bipartite graph on $T \cup N_G(T)$ with the edge set $\{\{v,w\} : v \in T, w \in N_G(T)\} \cap E(G)$. 
\item When $G$ has at least one odd cycle, a non-empty set $T \subset V(G)$ is said to be a \textit{fundamental set} if the following conditions are satisfied: 
\begin{itemize}
\item $B(T)$ is connected; 
\item $V(B(T))=V(G)$, or each connected component of $G \setminus V(B(T))$ contains an odd cycle. 
\end{itemize}
\item A graph $G$ is called bipartite if $V(G)$ can be decomposed into two sets $V_1,V_2$, called the partition, such that $E(G) \subset V_1 \times V_2$. 
\item When $G$ is a bipartite graph, a non-empty set $T \subset V(G)$ is said to be an \textit{acceptable set} if the following conditions are satisfied: 
\begin{itemize}
\item $B(T)$ is connected; 
\item $G \setminus V(B(T))$ is a connected graph with at least one edge. 
\end{itemize}
\end{itemize}

Given $i \in [d]$, let 
$$H_i=\{(x_1,\ldots,x_d) \in \RR^d : x_i=0\} \;\text{ and }\; H_i^{(+)}=\{(x_1,\ldots,x_d) \in \RR^d : x_i \geq 0\}.$$ 
Given an independent set $T \subset [d]$, let 
\begin{align*}
H_T&=\left\{(x_1,\ldots,x_d) \in \RR^d : \sum_{j \in N_G(T)}x_j - \sum_{i \in T}x_i = 0\right\} \;\text{ and }\\
H_T^{(+)}&=\left\{(x_1,\ldots,x_d) \in \RR^d : \sum_{j \in N_G(T)}x_j - \sum_{i \in T}x_i \geq 0\right\}.
\end{align*}
It is proved in \cite[Theorem 1.7]{OH98} that for any non-bipartite (resp., bipartite) graph $G$, each facet of $P_G$ is defined by 
a supporting hyperplane $H_i$ for some regular (resp., ordinary) vertex $i$ or $H_T$ for some fundamental (resp., acceptable) set. Let 
%\begin{align*}
%P_G=\bigcap_{\text{$i$ : regular vertices}}H_i^{(+)} \cap \bigcap_{\text{$T$ : fundamental sets}}H_T^{(+)}
%\end{align*}
\begin{align*}
\widetilde{\Psi}=\begin{cases}
\{H_i : \text{$i$ is a regular vertex}\} \cup \{ H_T : \text{$T$ is a fundamental set}\} & \text{if $G$ is non-bipartite}, \\
\{H_i : \text{$i$ is an ordinary vertex}\} \cup \{ H_T : \text{$T$ is an acceptable set}\} & \text{if $G$ is bipartite}. \end{cases}
\end{align*}
Although $\widetilde{\Psi}$ describes all supporting hyperplanes of the facets of $P_G$, 
it might happen that $H_i$ and $H_T$ define the same facet for some $i$ and $T$ if $G$ is bipartite. 

\begin{prop}
Let $G$ be a connected bipartite graph that has the partition $V(G)=V_1\sqcup V_2$. 
Then $\widetilde{\Psi}'=\{H_i : \text{$i$ is an ordinary vertex}\} \cup \{ H_T : \text{$T\subset V_1$ is an acceptable set}\}$ 
is the irredundant set of supporting hyperplanes of the facets of $P_G$.
\end{prop}
\begin{proof}
%As mentioned above, it is proved that for any bipartite graph $G$, each facet of $P_G$ is defined by a supporting hyperplane in $\widetilde{\Psi}$. 
We show that we can choose the set of accetable sets $T$ as a subset of $V_1$ and it is irredundant. 
It easily follows that either $T\subset V_1$ or $T\subset V_2$ holds if $T$ is acceptable. 
If $T\subset V_1$ is acceptable, then $B(T)$ and $G\setminus V(B(T))$ are connected with at least one edge. 
Therefore, set $T'=V_2\setminus N_G(T)$ and we can see that $B(T')=G\setminus V(B(T))$ and $G\setminus V(B(T'))=B(T)$, so $T'$ is an acceptable set contained in $V_2$. 
Conversely, if $S\subset V_2$ is acceptable, then there exists an acceptable set $S'\subset V_1$ with $S=V_2\setminus N_G(S')$. 
Thus, acceptable sets contained in $V_1$ one-to-one correspond to ones contained in $V_2$. 
Moreover, for an acceptable set $T\subset V_1$, $H_T$ and $H_{T'}$ define the same facet 
since $P_G$ is contained in the hyperplane $$\left\{(x_1,\ldots,x_d)\in \RR^d : \sum_{i\in V_1}x_i=\sum_{j\in V_2}x_j=1\right\}.$$ 
This implies that $\displaystyle \sum_{j\in N_G(T)}x_j-\sum_{i\in T}x_i=\sum_{i\in N_G(T')}x_i-\sum_{j\in T'}x_j.$ 
\end{proof}

\bigskip

%%%%%%%%%%%%%%%%%%%%%%%%%%%%%%%%%%%%%%%%%%%%%%%%%%%%%%%%%%%%%%%%%%%%%%%%%%%%%%%%%%%%%%%%%%%%%%%%%%%%%%%%%%%%%%%%%%%%%%%%%%%%%%%%%%%%
%%%%%%%%%%%%%%%%%%%%%%%%%%%%%%%%%%%%%%%%%%%%%%%%%%%%%%%%%%%%%%%%%%%%%%%%%%%%%%%%%%%%%%%%%%%%%%%%%%%%%%%%%%%%%%%%%%%%%%%%%%%%%%%%%%%%
%%%%%%%%%%%%%%%%%%%%%%%%%%%%%%%%%%%%%%%%%%%%%%%%%%%%%%%%%%%%%%%%%%%%%%%%%%%%%%%%%%%%%%%%%%%%%%%%%%%%%%%%%%%%%%%%%%%%%%%%%%%%%%%%%%%%
%%%%%%%%%%%%%%%%%%%%%%%%%%%%%%%%%%%%%%%%%%%%%%%%%%%%%%%%%%%%%%%%%%%%%%%%%%%%%%%%%%%%%%%%%%%%%%%%%%%%%%%%%%%%%%%%%%%%%%%%%%%%%%%%%%%%
%%%%%%%%%%%%%%%%%%%%%%%%%%%%%%%%%%%%%%%%%%%%%%%%%%%%%%%%%%%%%%%%%%%%%%%%%%%%%%%%%%%%%%%%%%%%%%%%%%%%%%%%%%%%%%%%%%%%%%%%%%%%%%%%%%%%
\section{Class groups of toric rings and their torsionfreeness}\label{sec:torsionfreee}

In this section, we discuss descriptions of the class groups of Hibi rings, stable set rings and edge rings 
in terms of the underlying posets or graphs. As their corollary, we see that their class groups are torsionfree.

\subsection{Class groups of Hibi rings}
First, we consider the class groups of Hibi rings. 
In \cite{HHN}, the description of class groups of Hibi rings is provided, which we describe below. 

Let $\Pi$ be a poset and let $|\Pi|=d$. 
Let $\widehat{\Pi}=\Pi \sqcup \{\hat{0},\hat{1}\}$, 
where $\hat{0}$ (resp. $\hat{1}$) is a new minimal (resp. maximal) element not belonging to $\Pi$. 
Thus, $|\widehat{\Pi}|=d+2$. 
Let $n$ be the number of the edges of the Hasse diagram of $\widehat{\Pi}$. 
Then it is proved in \cite{HHN} that \begin{align}\label{eq:Hibi}\Cl(\kk[\Pi]) \cong \ZZ^{n-d-1}.\end{align} 
In particular, $\Cl(\kk[\Pi])$ is torsionfree.

\bigskip

\subsection{Class groups of stable set rings}
Next, we discuss the class groups of stable set rings of perfect graphs. 
\begin{prop}\label{prop:class_stab}
Let $G$ be a perfect graph with maximal cliques $Q_0, Q_1,\ldots, Q_n$. 
Then $\Cl(\kk[\Stab_G]) \cong \ZZ^n$. In particular, $\Cl(\kk[\Stab_G])$ is torsionfree. 
\end{prop}
\begin{proof}
As described in \eqref{facets:stab}, we have 
$$\Phi(\Stab_G)=\{\eb_i : i=1,\ldots,d\} \cup \left\{-\sum_{j\in Q_i}\eb_j +\eb_{d+1} : i=0,1,\ldots,n \right\},$$ 
where $\Phi$ is as in \eqref{eq:psi}. Thus, $|\Phi(\Stab_G)|-(d+1)=n$. 
By choosing $\eb_1,\ldots,\eb_d$ and $-\sum_{j\in Q_0}\eb_j+\eb_{d+1}$ from $\Phi(\Stab_G)$, 
we obtain that $\det(\eb_1,\ldots,\eb_d,-\sum_{j\in Q_0}\eb_j+\eb_{d+1}) = 1$. 
Hence, Lemma~\ref{lem class group} implies that $\Cl(\kk[\Stab_G]) \cong \ZZ^n$. 
\end{proof}

\bigskip
\subsection{Class groups of edge rings}
Finally, we discuss the class groups of edge rings of connected graphs satisfying the odd cycle condition.

Let $\Psi=\Psi_r \cup \Psi_f$ (resp., $\Psi=\Psi_o \cup \Psi_a$) if $G$ is non-bipartite (resp., bipartite), where \begin{align*}
&\Psi_r=\{\ell_{H_i} : i \text{ is a regular vertex}\} \text{, } \Psi_f=\{\ell_{H_T} : T \text{ is a fundamental set}\}, \\
&\Psi_o=\{\ell_{H_i} : i \text{ is an ordinary vertex}\} \text{ and } \Psi_a=\{\ell_{H_T} : T \text{ is an acceptable set}\}.
\end{align*}

In particular, if $G$ is a connected graph, then we obtain that 
\begin{align*}
&\Psi_r=\{\eb_i : i \text{ is a regular vertex}\} \text{, } \Psi_o=\{\eb_i : i \text{ is an ordinary vertex}\}, \\
&\Psi_f=\begin{dcases}
\left\{\ell_{H_T}=\sum_{j\in N_G(T)}\eb_j-\sum_{i\in T}\eb_i : T \text{ is a fundamental set}\right\} \text{ if $V(B(T))\ne V(G)$}, \\
\left\{\ell_{H_T}=\frac{1}{2}\left(\sum_{j\in N_G(T)}\eb_j-\sum_{i\in T}\eb_i\right) : T \text{ is a fundamental set}\right\} \text{ if $V(B(T))=V(G)$}, 
\end{dcases}\\
&\Psi_a=\left\{\ell_{H_T}=\sum_{j\in N_G(T)}\eb_j-\sum_{i\in T}\eb_i : T\subset V_1 \text{ is an acceptable set}\right\}. 
\end{align*}
Note that $\frac{1}{2}$ appears in the case of $V(B(T))=V(G)$ since $\langle\ell_{H_T},\rho(e) \rangle=0$ or $2$ in this case, 
while $\langle\ell_{H_T},\rho(e) \rangle=1$ happens otherwise. 

\medskip

Let us fix some notation on graph theory. 
For a graph $G$, a \textit{path} is a non-empty subgraph $P=p_0 p_1 \cdots p_k$ of $G$ 
on the vertex set $V(P)=\{p_0, p_1,\ldots,p_k\}$ with the edge set $E(P)=\{ \{p_0, p_1\}, \{p_1,p_2\},\ldots,\{p_{k-1}, p_k \} \}$, 
where $p_i$'s are all distinct. Then we say that the vertices $p_0$ and $p_k$ are \textit{connected by $P$} and 
$p_0$ and $p_k$ are called its \textit{end vertices} or \textit{ends}. 
The \textit{interior} of $P$, denoted by $P^{\circ}$, is the vertices except for $p_0,p_k$. 
A \textit{cycle} is a non-empty subgraph $C=p_0 p_1 \cdots p_k p_0$ on the vertex set $V(C)=\{p_0, p_1,\ldots,p_k\}$ 
with the edge set $E(C)=\{ \{p_0, p_1\}, \{p_1,p_2\},\ldots,\{p_{k-1}, p_k \}, \{p_k, p_0 \} \}$, where $p_i$'s are all distinct. 
%The \textit{length} of a cycle is its number of edges (or vertices); an \textit{odd cycle} (resp., \textit{even cycle}) is a cycle of odd (resp., even) length. 

For an edge $e$ which is not an edge of a path $P$ (resp. a cycle $C$), 
$e$ is called a \textit{chord} of $P$ (resp. $C$) if $e$ joins two vertices of $P$ which are not end vertices (resp., two vertices of $C$). 
A path (resp., cycle) which has no chord is called \textit{primitive}. 

A \textit{block} of a graph $G$ means a $2$-connected component of $G$. Namely, a block contains no cut vertex. 
%It may contain cut vertices of $G$. 
Let $A$ denote the set of cut vertices of $G$, and $\calB$ the set of its blocks. 
We then have a natural bipartite graph on the vertex set $A\sqcup \calB$ 
with the edge set $\{\{a,B\} : a\in B\text{ for }a\in A \text{ and } B\in \calB\}$. 
We call this bipartite graph the \textit{block graph} of $G$, denoted by $\Block(G)$. 
Note that $\Block(G)$ is a tree if $G$ is connected. 

\bigskip

The following lemma will be used for the proofs of our results in many times. 
\begin{lem}\label{fundamental}
Let $G$ be a non-bipartite connected graph. 
\begin{itemize}
\item[(1)] Suppose that $S$ is an independent set of $G$ such that $B(S)$ is connected. 
Then there exists a fundamental set $T$ such that $S\subset T$ and $V(B(T))=V(G)$. 
\item[(2)] Let $C=p_0 p_1 \cdots p_{2k} p_0$ be a primitive odd cycle of length $2k+1$ in $G$. 
Then, for each $i=0,\ldots,2k$, there exists a fundamental set $T_i$ such that 
$E(C)\setminus \{p_i,p_{i+1}\}\subset E(B(T_i))$ and $\{p_i,p_{i+1}\} \notin E(B(T_i))$, 
where $p_{2k+1}=p_0$. In particular, $G$ has at least $2k+1$ fundamental sets.
\end{itemize}
\end{lem}
\begin{proof}
(1) If $V(G)= V(B(S))$, then $S$ itself satisfies the required property. Suppose that $V(B(S)) \subsetneq V(G)$. 
Then there exists $v\in V(G)\setminus V(B(S))$ such that $v$ and $w$ are adjacent for some $w\in N_G(S)$ since $G$ is connected. 
Thus, $S'=S\cup \{v\}$ is an independent set and $B(S')$ is connected. 
We repeat this application and we eventually obtain $S'$ which satisfies that $B(S')$ is connected and $V(B(S'))=V(G)$, as required. 

\noindent
(2) Fix $i=0$. 
By setting $S=\{p_2,p_4,\ldots,p_{2k}\}$, we can see that $S$ is an independent set since $C$ is primitive and 
$B(S)$ is a connected graph with $E(C)\setminus \{p_0,p_1\}\subset E(B(S))$ and $\{p_i,p_{i+1}\} \notin E(B(S))$. 
A proof directly follows from (1). 
\end{proof}

\begin{rem}\label{remark}
Let $G$ be a non-bipartite connected graph with a cut vertex $v$ and let $C_1,\ldots,C_n$ be connected components of $G\setminus v$. 
For $i=1,\ldots,n$, let $G_i=C_i+v$. Suppose that $G_1$ contains an odd cycle and let $T$ be a fundamental set in $G_1$. 

If $v\in V(B(T))$, then there exists a fundamental set $T'$ in $G$ with $V(B(T'))=V(B(T)) \cup \bigcup_{i=2}^n V(G_i)$. 
We can construct it similarly to Lemma~\ref{fundamental} (1). %by restricting on $G_2,\ldots,G_n$. 
We call this fundamental set $T'$ an \textit{induced fundamental set} of $T$. 
Note that an induced fundamental set is not unique but for distinct fundamental sets $T$ and $S$ in $G_1$ with $v\in V(B(T)) \cap V(B(S))$, 
their induced fundamental sets are distinct. Moreover, if $v$ is a regular vertex in $G$, 
then there exists a fundamental set $T''$ in $G$ with $V(B(T''))=\bigcup_{i=2}^n V(G_i)$ in the same way. 
We regard $T''$ as an induced fundamental set of the empty set although the empty set is not fundamental.

If $v\notin V(B(T))$, then $T$ is also a fundamental set in $G$. 
Therefore, we can observe that $|\Psi_f (G)|\ge |\Psi_f (G_1)|$ and $|\Psi_f (G)|\ge |\Psi_f (G_1)|+1$ if $v$ is regular in $G$. 
\end{rem}

\begin{lem}\label{lem:linear_depend}
Let $G$ be a graph. 
\begin{itemize}
\item[(1)] Let $e_1,\ldots,e_{2k}$ be the edges of an even cycle in $G$. 
Then $$w_{\rho(e_1)},\ldots,w_{\rho(e_{2k})}$$ are linearly dependent. \\
\item[(2)] Let $C$ and $C'$ be two odd cycles and let $e_1,\ldots,e_{2k+1}$ (resp. $e_1',\ldots,e_{2k'+1}'$) be the edges of $C$ (resp. $C'$). 
\begin{itemize}
\item[(2-1)] Assume that $C$ and $C'$ have a unique common vertex. Then 
$$w_{\rho(e_1)},\ldots,w_{\rho(e_{2k+1})},w_{\rho(e_1')},\ldots,w_{\rho(e_{2k'+1}')}$$ are linearly dependent. 
\item[(2-2)] Assume that $C$ and $C'$ have no common vertex but there is a path whose edges are $f_1,\ldots,f_m$ between $C$ and $C'$ connecting them. 
Then $$w_{\rho(e_1)},\ldots,w_{\rho(e_{2k+1})},w_{\rho(e_1')},\ldots,w_{\rho(e_{2k'+1}')},w_{\rho(f_1)},\ldots,w_{\rho(f_m)}$$ 
are linearly dependent. 
\end{itemize}
\end{itemize}
\end{lem}
\begin{proof}
(1) We see that \begin{align*}
\sum_{i=1}^{2k}(-1)^iw_{\rho(e_i)}&=\sum_{i=1}^{2k}(-1)^i\sum_{\ell \in \Psi} \langle \ell, \rho(e_i) \rangle \eb_{\ell} 
=\sum_{\ell \in \Psi} \langle \ell, \sum_{i=1}^{2k}(-1)^i \rho(e_i) \rangle \eb_{\ell} 
=\sum_{\ell \in \Psi} \langle \ell, {\bf 0} \rangle \eb_{\ell}={\bf 0}. 
\end{align*}

\noindent
(2) In the case (2-1), let $e_1 \cap e_{2k+1} \cap e_1' \cap e_{2k'+1}'$ be the unique common vertex of $C$ and $C'$. 
In the case (2-2), let $P$ be the path consisting of $f_1,\ldots,f_m$ 
which connects the vertex $e_1 \cap e_{2k+1}$ of $C$ and $e_1' \cap e_{2k'+1}'$ of $C'$. 
Then we see the following: 
\begin{align*}
&\sum_{i=1}^{2k+1}(-1)^iw_{\rho(e_i)}-\sum_{i=1}^{2k'+1}(-1)^iw_{\rho(e_i')}={\bf 0}; \\
&\sum_{i=1}^{2k+1}(-1)^iw_{\rho(e_i)}-\sum_{i=1}^{2k'+1}(-1)^iw_{\rho(e_i')}-2\sum_{j=1}^m(-1)^jw_{\rho(f_j)}={\bf 0} \text{ if $m$ is even}; \\
&\sum_{i=1}^{2k+1}(-1)^iw_{\rho(e_i)}+\sum_{i=1}^{2k'+1}(-1)^iw_{\rho(e_i')}-2\sum_{j=1}^m(-1)^jw_{\rho(f_j)}={\bf 0} \text{ if $m$ is odd}. 
\end{align*}
\end{proof}

\begin{prop}[{cf. \cite[Proposition 10.1.48]{Villa}}]\label{product}
Let $G$ be a graph.
\begin{itemize}
\item[(1)] Let $G_1,\ldots,G_n$ be the connected components of $G$. 
Then we have $\kk[G]\cong \kk[G_1]\otimes \cdots \otimes \kk[G_n]$. 
Therefore, $\Cl(\kk[G]) \cong \Cl(\kk[G_1]) \oplus \cdots \oplus \Cl(\kk[G_n])$. 
\item[(2)] Suppose that $G$ is connected and let $B_1,\ldots, B_m$ be the blocks of $G$. 
If there is at most one non-bipartite block among them, then we have $\kk[G]\cong \kk[B_1]\otimes \cdots \otimes \kk[B_m]$. 
Therefore, $\Cl(\kk[G]) \cong \Cl(\kk[B_1]) \oplus \cdots \oplus \Cl(\kk[B_m])$.
\end{itemize}
\end{prop}

Now, we are ready to discuss the description of $\Cl(\kk[G])$ and show its torsionfreeness for $G$ satisfying the odd cycle condition. 
\begin{thm}\label{thm:class_edge}
Let $G$ be a connected graph satisfying the odd cycle condition. 
Then $\Cl(\kk[G]) \cong \ZZ^{|\Psi|-\dim \kk[G]}$. In particular, $\Cl(\kk[G])$ is torsionfree.
\end{thm}
\begin{proof}
By proposition~\ref{prop:class_group}, it is enough to show that $\rank \calM=\dim \kk[G]$ and $d_1=\cdots =d_s=1$.

\medskip

\noindent\underline{The case where $G$ is bipartite}: 

We may assume that $G$ is $2$-connected by Proposition~\ref{product}. Take a spanning tree $T$ of $G$. 
For any $e'\in E(G)\setminus E(T)$, the subgraph $T'$ obtained by adding $e'$ to $T$ has an even cycle containing $e'$. 
We see that $w_{\rho(e)}$'s for $e\in E(T')$ are linearly dependent by Lemma~\ref{lem:linear_depend}, 
so we can erase the columns corresponding to the edges $e'\in E(G)\setminus E(T)$ in $\calM$ by using $e\in T$. 
Moreover, we consider the row corresponding to  (the supporting hyperplane of) the ordinary vertex $v$ whose degree is $1$ in $T$. 
Since $G$ is $2$-connected, i.e., every vertex in $G$ is ordinary, the entry corresponding to the edge $e_0$ which joins $v$ is $1$ 
and the other entries are all $0$ in the row. Therefore, $w_{\rho(e_0)}$ can be transformed into a unit vector. 
We repeat this transformation for $T\setminus v$. Then we can see that $w_{\rho(e)}$'s for $e\in E(T)$ are linearly independent, 
that is, $\rank \calM=|T|=d-1=\dim \kk[G]$ and $d_1=\cdots=d_s=1$. 

\medskip

\noindent\underline{The case where $G$ is non-bipartite}: 

Let $B_1,\ldots,B_m$ be the blocks of $G$. We prove the assertion by induction on $m$. 

Let $G'$ be a connected subgraph $G'$ of $G$ satisfying the following properties: 
\begin{itemize}
\item $G'$ is a spanning subgraph of $G$;
\item $G'$ has $d$ edges;
\item $G'$ has exactly one primitive odd cycle $C=p_0\cdots p_{2k} p_0$. 
\end{itemize}

In the case $m=1$, for any $e'\in E(G)\setminus E(G')$, consider the subgraph $G''$ obtained by adding $e'$ to $G'$. 
Then $G''$ satisfies one of the following conditions: 
\begin{itemize}
\item[(i)] $G''$ contains an even cycle; 
\item[(ii)] $G''$ contains two odd cycles and they have a unique common vertex; 
\item[(iii)] $G''$ contains two odd cycles $C'$ and $C''$ with no common vertex but there is a path between $C'$ and $C''$ connecting them. 
\end{itemize}
We can see that $w_{\rho(e)}$'s for $e\in E(G'')$ are linearly dependent by Lemma~\ref{lem:linear_depend}. 
Moreover, since $G$ is $2$-connected, i.e., every vertex in $G$ except for $V(C)$ is regular, 
$w_{\rho(e)}$'s for $e\in E(G')\setminus E(C)$ can be transformed into a unit vector by the same discussions above. 
For $\{p_i,p_{i+1}\}$ ($i=0,\ldots,2k$), take a fundamental set $T$ satisfying Lemma~\ref{fundamental} (2). 
Then the entry corresponding to the edge $\{p_i,p_{i+1}\}$ is $1$ and the other entries are all $0$ in the row corresponding to (the supporting hyperplane of) the fundamental set $T$. 
Thus, $w_{\rho(\{p_i,p_{i+1}\})}$ can be transformed into a unit vector. 
Hence, we conclude that $\rank \calM=|G'|=d=\dim \kk[G]$ and $d_1=\cdots=d_s=1$.

Let $m\ge 2$. Then there exists $B_i$ containing a unique primitive odd cycle $C$ such that $G'_{V(B_j)}$ is a tree for $j\ne i$. 
We may assume that $i=1$. Note that all vertices in $G$ are regular on $G$ except for cut vertices of $G$ and $p_0,\ldots,p_{2k}$. 
Then we can find a cut vertex $v$ of $G$ such that the subgraph $\Block(G)\setminus v$ of $\Block(G)$ 
has a unique connected component containing $B_1$ and the other components are isolated vertices; 
these are blocks $B_{i_1},\ldots,B_{i_l}$ such that $B'_{i_j}=G'_{V(B_{i_j})}$ are trees. 
Since every vertex in $\bigcup_{j\in [l]}V(B_{i_j})$ is regular except for $v$, 
$w_{\rho(e)}$'s for $e\in \bigcup_{j\in [l]}E(B'_{i_j})$ can be transformed into a unit vector. 
Let $H=G\setminus \left(\bigcup_{j\in [l]}V(B_{i_j})\setminus \{v\}\right)$. 
As mentioned in Remark~\ref{remark}, if a vertex $u \ne v$ on $H$ is regular, then $u$ is also regular on $G$, 
and if $S$ is a fundamental set on $H$, then $S$ or an induced fundamental $S'$ is fundamental on $G$. 
Although $v$ is not regular on $G$, it might happen that $v$ is regular on $H$. 
If $v$ is regular on $H$, we can take an induced fundamental set $U$ of the empty set on $G$. 
In the row corresponding to (the supporting hyperplane of) the fundamental set $U$, 
the entries corresponding to the edges joining $v$ on $H$ is $1$ and the other entries are all $0$. 
Thus, we can regard a fundamental set $U$ on $G$ as a regular vertex on $H$. 
Therefore, we can see that $w_{\rho(e)}$'s for $e\in E(H)\cap E(G')$ can be transformed into unit vectors by induction. 
\end{proof}

\bigskip

%%%%%%%%%%%%%%%%%%%%%%%%%%%%%%%%%%%%%%%%%%%%%%%%%%%%%%%%%%%%%%%%%%%%%%%%%%%%%%%%%%%%%%%%%%%%%%%%%%%%%%%%%%%%%%%%%%%%%%%%%%%%%%%%%%%%
%%%%%%%%%%%%%%%%%%%%%%%%%%%%%%%%%%%%%%%%%%%%%%%%%%%%%%%%%%%%%%%%%%%%%%%%%%%%%%%%%%%%%%%%%%%%%%%%%%%%%%%%%%%%%%%%%%%%%%%%%%%%%%%%%%%%
%%%%%%%%%%%%%%%%%%%%%%%%%%%%%%%%%%%%%%%%%%%%%%%%%%%%%%%%%%%%%%%%%%%%%%%%%%%%%%%%%%%%%%%%%%%%%%%%%%%%%%%%%%%%%%%%%%%%%%%%%%%%%%%%%%%%
%%%%%%%%%%%%%%%%%%%%%%%%%%%%%%%%%%%%%%%%%%%%%%%%%%%%%%%%%%%%%%%%%%%%%%%%%%%%%%%%%%%%%%%%%%%%%%%%%%%%%%%%%%%%%%%%%%%%%%%%%%%%%%%%%%%%
%%%%%%%%%%%%%%%%%%%%%%%%%%%%%%%%%%%%%%%%%%%%%%%%%%%%%%%%%%%%%%%%%%%%%%%%%%%%%%%%%%%%%%%%%%%%%%%%%%%%%%%%%%%%%%%%%%%%%%%%%%%%%%%%%%%%
\section{Toric rings whose class groups are rank $1$ or $2$}\label{sec:small}

In this section, we provide a characterization of posets or graphs 
whose associated toric rings have their class groups $\ZZ$ or $\ZZ^2$. 

\subsection{Hibi rings with small class groups}
We define four posets as follows.
\begin{itemize}
\item[(i)] For $s_1,s_2 \in \ZZ_{>0}$, 
let $\Pi_1(s_1,s_2)=\{p_1,\ldots,p_{s_1},p_{s_1+1},\ldots,p_{s_1+s_2}\}$ be the poset equipped with the partial orders 
$p_1 \prec \cdots \prec p_{s_1}$ and $p_{s_1+1} \prec \cdots \prec p_{s_1+s_2}$. Figure~\ref{poset1} shows the Hasse diagram of $\Pi_1(s_1,s_2)$. 
\item[(ii)] 
For $s_1, s_2, s_3 \in \ZZ_{>0}$ and $t \in \ZZ_{\geq 0}$, 
let $\Pi_2(s_1,s_2,s_3,t)=\{p_1,\ldots,p_d\}$ $(d=s_1+s_2+s_3+t)$ be the poset equipped with the partial orders
\begin{itemize} 
\item $p_1 \prec \cdots \prec p_t$, 
\item $p_t \prec p_{t+1} \prec \cdots \prec p_{t+s_1}$ and $p_t \prec p_{t+s_1+1} \prec \cdots \prec p_{t+s_1+s_2}$ 
($p_1 \prec \cdots \prec p_{s_1}$ and $p_{s_1+1} \prec \cdots \prec p_{s_1+s_2}$ if $t=0$) and 
\item $p_{t+s_1+s_2+1} \prec \cdots \prec p_d$. 
\end{itemize}
Figure~\ref{poset2} shows the Hasse diagram of $\Pi_2(s_1,s_2,s_3,t)$ and Figure \ref{poset20} is the case $t=0$. 
\item[(iii)] 
Moreover, for $s_1,s_2,t_1,t_2 \in \ZZ_{>0}$ and $t_3 \in \ZZ_{\geq 0}$, let 
$\Pi_3(s_1,s_2,t_1,t_2,t_3)=\{p_1,\ldots,p_d\}$ $(d=s_1+s_2+t_1+t_2+t_3)$ be the poset equipped with the partial orders
\begin{itemize}
\item $p_1 \prec \cdots \prec p_{t_1} \prec p_{t_1+1} \cdots \prec p_{t_1+s_1}$, 
\item $p_{t_1+s_1+1} \prec \cdots \prec p_{t_1+s_1+s_2} \prec p_{s_1+t_1+s_2+1} \cdots \prec p_{t_1+s_1+s_2+t_2}$ and 
\item $p_{t_1} \prec p_{t_1+s_1+s_2+t_2+1}\cdots \prec p_d \prec p_{t_1+s_1+s_2+1}$. 
\end{itemize}
Figure~\ref{poset3} shows the Hasse diagram of $\Pi_3(s_1,s_2,t_1,t_2,t_3)$. 
\item[(iv)] Furthermore, for $s_1,s_2,t_1,t_2 \in \ZZ_{>0}$,  let 
$\Pi_4(s_1,s_2,t_1,t_2)=\{p_1,\ldots,p_{d+1}\}$ $(d=s_1+s_2+t_1+t_2)$ be the poset equipped with the partial orders
\begin{itemize}
\item $p_1 \prec \cdots \prec p_{t_1} \prec p_{d+1}$, $p_{t_1+1} \prec \cdots \prec p_{t_1+t_2} \prec p_{d+1}$ and 
\item $p_{d+1} \prec p_{t_1+t_2+1} \prec \cdots \prec p_{t_1+t_2+s_1}$, $p_{d+1} \prec p_{t_1+t_2+s_1+1} \prec \cdots \prec p_d$.
\end{itemize}
Figure~\ref{poset4} shows the Hasse diagram of $\Pi_4(s_1,s_2,t_1,t_2)$. 
\end{itemize}

%%%%%%%%%%%%%%%%%%%%%
\begin{figure}[h]%1
{\scalebox{0.8}{
\begin{minipage}{1.0\columnwidth}
\centering
{\scalebox{0.8}{
\begin{tikzpicture}[line width=0.05cm]

\coordinate (N11) at (0,0); 
\coordinate (N12) at (0,1); 
\coordinate (N13) at (0,2); 
\coordinate (N15) at (0,4); 
\coordinate (N16) at (0,5);

\coordinate (N21) at (2,0); 
\coordinate (N22) at (2,1); 
\coordinate (N23) at (2,2); 
\coordinate (N25) at (2,4); 
\coordinate (N26) at (2,5);

%edge
\draw  (N11)--(N12); 
\draw  (N12)--(N13); 
\draw  (N13)--(0,2.5); 
\draw  (0,3.5)--(N15); 
\draw  (N15)--(N16); 
\draw[dotted]  (0,2.8)--(0,3.2);

\draw  (N21)--(N22);
\draw  (N22)--(N23); 
\draw  (N23)--(2,2.5);
\draw  (2,3.5)--(N25); 
\draw  (N25)--(N26); 
\draw[dotted]  (2,2.8)--(2,3.2);

%node

\draw [line width=0.05cm, fill=white] (N11) circle [radius=0.15]; %node[below right] {\Large $p_1$}; 
\draw [line width=0.05cm, fill=white] (N12) circle [radius=0.15]; %node[below left] {\Large $p_2$};
\draw [line width=0.05cm, fill=white] (N13) circle [radius=0.15]; %node[below left] {\Large $p_3$}; 
\draw [line width=0.05cm, fill=white] (N15) circle [radius=0.15]; %node[below left] {\Large $p_{s_1-1}$};
\draw [line width=0.05cm, fill=white] (N16) circle [radius=0.15]; %node[below right] {\Large $p_{s_1}$};
\draw [line width=0.05cm, fill=white] (N21) circle [radius=0.15]; %node[below right] {\Large $p_{s_1+1}$}; 
\draw [line width=0.05cm, fill=white] (N22) circle [radius=0.15]; %node[below left] {\Large $p_{s_1+2}$};
\draw [line width=0.05cm, fill=white] (N23) circle [radius=0.15]; %node[below left] {\Large $p_{s_1+3}$}; 
\draw [line width=0.05cm, fill=white] (N25) circle [radius=0.15]; %node[below left] {\Large $p_{s_1+s_2-1}$};
\draw [line width=0.05cm, fill=white] (N26) circle [radius=0.15]; %node[below right] {\Large $p_{s_1+s_2}$};

\node at (4,0) {$$}; 
\draw [line width=0.015cm, decorate,decoration={brace,amplitude=10pt}](-0.25,0) -- (-0.25,5) %node[midway,xshift=-1.3cm,yshift=0.2cm] {\Large $s_1$}
node[midway,xshift=-0.8cm] {\Large $s_1$}; 
\draw [line width=0.015cm, decorate,decoration={brace,amplitude=10pt}](1.75,0) -- (1.75,5) 
node[black,midway,xshift=-0.8cm] {\Large $s_2$};
%node[black,midway,xshift=1.3cm,yshift=-0.2cm] {\Large vertices}; 
%\draw [line width=0.03cm, decorate,decoration={brace,amplitude=10pt,mirror}](0,-1.1) -- (2,-1.1) node[black,midway,yshift=-0.7cm] {\Large $t$ columns}; 

\end{tikzpicture}
} }
\caption{The poset $\Pi_1$}
\label{poset1}
\end{minipage} 
}}
{\scalebox{0.8}{
\begin{minipage}{0.46\columnwidth}
\centering
{\scalebox{0.7}{
\begin{tikzpicture}[line width=0.05cm]

\coordinate (N14) at (0,3); 
\coordinate (N15) at (0,4); 
\coordinate (N16) at (0,5); 
\coordinate (N17) at (0,6);

\coordinate (N24) at (2,3); 
\coordinate (N25) at (2,4); 
\coordinate (N26) at (2,5); 
\coordinate (N27) at (2,6);

\coordinate (N31) at (4,0); 
\coordinate (N32) at (4,1); 
\coordinate (N33) at (4,2); 
\coordinate (N35) at (4,4); 
\coordinate (N36) at (4,5); 
\coordinate (N37) at (4,6);

\coordinate (Nt1) at (1,0); 
\coordinate (Nt3) at (1,2); 

%edge

\draw  (Nt3)--(N14);
\draw  (N14)--(N15); 
\draw  (N15)--(0,4.5);
\draw[dotted]  (0,4.8)--(0,5.3);
\draw  (0,5.5)--(N17);

\draw  (Nt3)--(N24); 
\draw  (N24)--(N25);
\draw  (N25)--(2,4.5);
\draw[dotted]  (2,4.8)--(2,5.3); 
\draw  (2,5.5)--(N27);

\draw  (Nt1)--(1,0.5); 
\draw[dotted]  (1,0.8)--(1,1.3);
\draw  (1,1.5)--(Nt3);

\draw  (N31)--(N32); 
\draw  (N32)--(N33); 
\draw  (N33)--(4,2.5); 
\draw  (4,3.5)--(N35); 
\draw  (N35)--(N36); 
\draw  (N36)--(N37);
\draw[dotted]  (4,2.8)--(4,3.2);

%node
\draw [line width=0.05cm, fill=white] (Nt1) circle [radius=0.15];% node[below right] {\Large $p_1$};
\draw [line width=0.05cm, fill=white] (Nt3) circle [radius=0.15];% node[below right] {\Large $p_t$}; 

\draw [line width=0.05cm, fill=white] (N14) circle [radius=0.15]; %node[right] {\Large $p_{t+1}$};
\draw [line width=0.05cm, fill=white] (N15) circle [radius=0.15];
\draw [line width=0.05cm, fill=white] (N17) circle [radius=0.15];% node[] at (0.8,3) {\Large $p_{t+s_1}$};

\draw [line width=0.05cm, fill=white] (N24) circle [radius=0.15]; 
\draw [line width=0.05cm, fill=white] (N25) circle [radius=0.15];
\draw [line width=0.05cm, fill=white] (N27) circle [radius=0.15];% node[below right] {\Large $p_{t+s_1+s_2}$};

\draw [line width=0.05cm, fill=white] (N31) circle [radius=0.15];% node[below right] {\Large $p_{t+s_1+s_2+1}$}; 
\draw [line width=0.05cm, fill=white] (N32) circle [radius=0.15]; %node[below right] {\Large $p_{t+s_1+s_2+2}$}; 
\draw [line width=0.05cm, fill=white] (N33) circle [radius=0.15]; %node[below right] {\Large $p_{t+s_1+s_2+3}$}; 
\draw [line width=0.05cm, fill=white] (N35) circle [radius=0.15]; %node[below right] {\Large $p_{t+s_1+s_2+s_3-2}$}; 
\draw [line width=0.05cm, fill=white] (N36) circle [radius=0.15]; %node[below right] {\Large $p_{t+s_1+s_2+s_3-1}$}; 
\draw [line width=0.05cm, fill=white] (N37) circle [radius=0.15];% node[below right] {\Large $p_{t+s_1+s_2+s_3}$};

\node at (6,0) {$$}; 
\draw [line width=0.015cm, decorate,decoration={brace,amplitude=10pt}](-0.25,3) -- (-0.25,6) %node[midway,xshift=-1.3cm,yshift=0.2cm] {\Large $s_1$}
node[midway,xshift=-0.8cm] {\Large $s_1$}; 
\draw [line width=0.015cm, decorate,decoration={brace,amplitude=10pt}](1.75,3) -- (1.75,6) node[midway,xshift=-0.8cm] {\Large $s_2$};
%node[black,midway,xshift=1.3cm,yshift=-0.2cm] {\Large vertices}; 
\draw [line width=0.015cm, decorate,decoration={brace,amplitude=10pt}](3.75,0) -- (3.75,6) node[black,midway,xshift=-0.8cm] {\Large $s_3$}; 
\draw [line width=0.015cm, decorate,decoration={brace,amplitude=10pt}](0.75,0) -- (0.75,2) node[black,midway,xshift=-0.8cm] {\Large $t$};

\end{tikzpicture} }}
\caption{The poset $\Pi_2$}
\label{poset2}
\end{minipage}
\begin{minipage}{0.53\columnwidth}
\centering
{\scalebox{0.8}{
\begin{tikzpicture}[line width=0.05cm]

\coordinate (N11) at (0,0); 
\coordinate (N12) at (0,1); 
\coordinate (N13) at (0,2); 
\coordinate (N15) at (0,4); 
\coordinate (N16) at (0,5);

\coordinate (N21) at (2,0); 
\coordinate (N22) at (2,1); 
\coordinate (N23) at (2,2); 
\coordinate (N25) at (2,4); 
\coordinate (N26) at (2,5);

\coordinate (N31) at (4,0); 
\coordinate (N32) at (4,1); 
\coordinate (N33) at (4,2); 
\coordinate (N35) at (4,4); 
\coordinate (N36) at (4,5);

%edge
\draw  (N11)--(N12); 
\draw  (N12)--(N13); 
\draw  (N13)--(0,2.5); 
\draw  (0,3.5)--(N15); 
\draw  (N15)--(N16); 
\draw[dotted]  (0,2.8)--(0,3.2);

\draw  (N21)--(N22);
\draw  (N22)--(N23); 
\draw  (N23)--(2,2.5);
\draw  (2,3.5)--(N25); 
\draw  (N25)--(N26); 
\draw[dotted]  (2,2.8)--(2,3.2);

\draw  (N31)--(N32); 
\draw  (N32)--(N33); 
\draw  (N33)--(4,2.5); 
\draw  (4,3.5)--(N35); 
\draw  (N35)--(N36); 
\draw[dotted]  (4,2.8)--(4,3.2);

%node

\draw [line width=0.05cm, fill=white] (N11) circle [radius=0.15];% node[below right] {\Large $p_1$}; 
\draw [line width=0.05cm, fill=white] (N12) circle [radius=0.15]; %node[below left] {\Large $p_2$};
\draw [line width=0.05cm, fill=white] (N13) circle [radius=0.15]; %node[below left] {\Large $p_3$}; 
\draw [line width=0.05cm, fill=white] (N15) circle [radius=0.15]; %node[below left] {\Large $p_{s_1-1}$};
\draw [line width=0.05cm, fill=white] (N16) circle [radius=0.15];% node[below right] {\Large $p_{s_1}$};

\draw [line width=0.05cm, fill=white] (N21) circle [radius=0.15];% node[below right] {\Large $p_{s_1+1}$}; 
\draw [line width=0.05cm, fill=white] (N22) circle [radius=0.15]; %node[below left] {\Large $p_{s_1+2}$};
\draw [line width=0.05cm, fill=white] (N23) circle [radius=0.15]; %node[below left] {\Large $p_{s_1+3}$}; 
\draw [line width=0.05cm, fill=white] (N25) circle [radius=0.15]; %node[below left] {\Large $p_{s_1+s_2-1}$};
\draw [line width=0.05cm, fill=white] (N26) circle [radius=0.15];% node[below right] {\Large $p_{s_1+s_2}$};

\draw [line width=0.05cm, fill=white] (N31) circle [radius=0.15];% node[below right] {\Large $p_{s_1+s_2+1}$};
\draw [line width=0.05cm, fill=white] (N32) circle [radius=0.15]; %node[below right] {\Large $p_{s_1+s_2+2}$};
\draw [line width=0.05cm, fill=white] (N33) circle [radius=0.15]; %node[below right] {\Large $p_{s_1+s_2+3}$}; 
\draw [line width=0.05cm, fill=white] (N35) circle [radius=0.15]; %node[below right] {\Large $p_{s_1+s_2+s_3-1}$};
\draw [line width=0.05cm, fill=white] (N36) circle [radius=0.15];% node[below right] {\Large $p_{s_1+s_2+s_3}$};

\node at (5.,0) {$$}; 

\draw [line width=0.015cm, decorate,decoration={brace,amplitude=10pt}](-0.25,0) -- (-0.25,5) %node[midway,xshift=-1.3cm,yshift=0.2cm] {\Large $s_1$}
node[midway,xshift=-0.8cm] {\Large $s_1$}; 
\draw [line width=0.015cm, decorate,decoration={brace,amplitude=10pt}](1.75,0) -- (1.75,5) node[black,midway,xshift=-0.8cm] {\Large $s_2$};
\draw [line width=0.015cm, decorate,decoration={brace,amplitude=10pt}](3.75,0) -- (3.75,5) node[black,midway,xshift=-0.8cm] {\Large $s_3$};
%node[black,midway,xshift=1.3cm,yshift=-0.2cm] {\Large vertices}; 
%\draw [line width=0.03cm, decorate,decoration={brace,amplitude=10pt,mirror}](0,-1.1) -- (2,-1.1) node[black,midway,yshift=-0.7cm] {\Large $t$ columns}; 

\end{tikzpicture}
} }
\caption{\; \\ The poset $\Pi_2$ with $t=0$}
\label{poset20}
\end{minipage}
}}
{\scalebox{0.8}{
\begin{minipage}{0.46\columnwidth}
\centering
{\scalebox{0.7}{
\begin{tikzpicture}[line width=0.05cm]

\coordinate (N11) at (0,0); 
\coordinate (N12) at (0,1); 
\coordinate (N13) at (0,2); 
\coordinate (N14) at (0,3);
\coordinate (N15) at (0,4); 
\coordinate (N16) at (0,5); 
\coordinate (N17) at (0,6);

\coordinate (N21) at (4,0); 
\coordinate (N22) at (4,1); 
\coordinate (N23) at (4,2); 
\coordinate (N24) at (4,3);
\coordinate (N25) at (4,4); 
\coordinate (N26) at (4,5); 
\coordinate (N27) at (4,6); 

\coordinate (Ntt1) at (1,2.5); 
\coordinate (Ntt2) at (3,3.5); 
\coordinate (Ntt3) at (4,2);

%edge
\draw  (N11)--(0,0.5); 
\draw  (0,1.5)--(N13); 
\draw  (N13)--(N14); 
\draw  (N14)--(0,3.5);
\draw  (0,4.5)--(N16);
\draw  (N16)--(N17);
\draw[dotted]  (0,0.8)--(0,1.2);
\draw[dotted]  (0,3.8)--(0,4.2);

\draw  (N21)--(N22);
\draw  (N22)--(4,1.5); 
\draw  (4,2.5)--(N24);
\draw  (N24)--(N25); 
\draw  (N25)--(4,4.5); 
\draw  (4,5.5)--(N27);
\draw[dotted]  (4,1.8)--(4,2.2);
\draw[dotted]  (4,4.8)--(4,5.2);

\draw  (N13)--(Ntt1); 
\draw  (Ntt2)--(N25); 
\draw  (Ntt1)--(1.5,2.75);
\draw  (2.5,3.25)--(Ntt2);
\draw[dotted]  (1.8,2.9)--(2.2,3.1);

%node

\draw [line width=0.05cm, fill=white] (N11) circle [radius=0.15];% node[below left] {\Large $p_1$}; 
\draw [line width=0.05cm, fill=white] (N13) circle [radius=0.15];% node[below left] {\Large $p_{t_1}$}; 
\draw [line width=0.05cm, fill=white] (N14) circle [radius=0.15]; %node[below left] {\Large $p_{s_1}$};
\draw [line width=0.05cm, fill=white] (N16) circle [radius=0.15];% node[below left] {\Large $p_{t_1+s_1}$};
\draw [line width=0.05cm, fill=white] (N17) circle [radius=0.15];% node[] at (-0.8,3) {\Large $p_{t_1+1}$};

\draw [line width=0.05cm, fill=white] (N21) circle [radius=0.15];% node[below right] {\Large $p_{t_1+s_1+1}$}; 
\draw [line width=0.05cm, fill=white] (N22) circle [radius=0.15]; %node[below left] {\Large $p_{s_1+2}$};
\draw [line width=0.05cm, fill=white] (N24) circle [radius=0.15];% node[below right] {\Large $p_{t_1+s_1+s_2+1}$};
\draw [line width=0.05cm, fill=white] (N25) circle [radius=0.15];% node[below right] {\Large $p_{t_1+s_1+s_2+t_2}$};
\draw [line width=0.05cm, fill=white] (N27) circle [radius=0.15];% node[below right] {\Large $p_{t_1+s_1+s_2}$};

\draw [line width=0.05cm, fill=white] (Ntt1) circle [radius=0.15] ; 
\draw [line width=0.05cm, fill=white] (Ntt2) circle [radius=0.15] ; 

\node at (4.5,0) {$$}; 
\draw [line width=0.015cm, decorate,decoration={brace,amplitude=10pt,mirror}](1.13,2.25) -- (3.13,3.25) %node[midway,xshift=-1.3cm,yshift=0.2cm] {\Large $s_1$}
node[midway,xshift=0.4cm,yshift=-0.6cm] {\Large $t_3$}; 

\draw [line width=0.015cm, decorate,decoration={brace,amplitude=10pt,mirror}](0.25,0) -- (0.25,2) node[black,midway,xshift=0.8cm] {\Large $t_1$};
\draw [line width=0.015cm, decorate,decoration={brace,amplitude=10pt,mirror}](0.25,3) -- (0.25,6) node[black,midway,xshift=0.8cm] {\Large $s_1$};

%node[black,midway,xshift=1.3cm,yshift=-0.2cm] {\Large vertices}; 

\draw [line width=0.015cm, decorate,decoration={brace,amplitude=10pt}](3.75,0) -- (3.75,3) 
node[midway,xshift=-0.8cm] {\Large $s_2$}; 
\draw [line width=0.015cm, decorate,decoration={brace,amplitude=10pt}](3.75,4) -- (3.75,6) 
node[midway,xshift=-0.8cm] {\Large $t_2$}; 

\end{tikzpicture} }}
\caption{The poset $\Pi_3$}
\label{poset3}
\end{minipage}
\begin{minipage}{0.49\columnwidth}
\centering
{\scalebox{0.7}{
\begin{tikzpicture}[line width=0.05cm]

\coordinate (N14) at (0,3); 
\coordinate (N15) at (0,4); 
\coordinate (N16) at (0,5); 
\coordinate (N17) at (0,6);

\coordinate (N24) at (2,3); 
\coordinate (N25) at (2,4); 
\coordinate (N26) at (2,5); 
\coordinate (N27) at (2,6);

\coordinate (Nt1) at (0,0); 
\coordinate (Nt2) at (0,1); 
\coordinate (Nt3) at (0,2); 

\coordinate (Nt5) at (2,0); 
\coordinate (Nt6) at (2,1); 
\coordinate (Nt7) at (2,2);

\coordinate (Ntt1) at (1,3); 

%edge

\draw (Nt1)--(0,0.5);
\draw[dotted] (0,0.8)--(0,1.2);
\draw (0,1.5)--(Nt3);

\draw (Nt5)--(2,0.5);
\draw[dotted] (2,0.8)--(2,1.2);
\draw (2,1.5)--(Nt7);

\draw (Nt3)--(Ntt1);
\draw (Nt7)--(Ntt1); 
\draw (Ntt1)--(N15);
\draw (Ntt1)--(N25);

\draw (N15)--(0,4.5);
\draw[dotted] (0,4.8)--(0,5.2);
\draw (0,5.5)--(N17);

\draw (N25)--(2,4.5);
\draw[dotted] (2,4.8)--(2,5.2);
\draw (2,5.5)--(N27);

%node
\draw [line width=0.05cm, fill=white] (Ntt1) circle [radius=0.15];% node[below right] {\Large $p_1$};

\draw [line width=0.05cm, fill=white] (N15) circle [radius=0.15];
\draw [line width=0.05cm, fill=white] (N17) circle [radius=0.15];% node[] at (0.8,3) {\Large $p_{t+s_1}$};

\draw [line width=0.05cm, fill=white] (N25) circle [radius=0.15];
\draw [line width=0.05cm, fill=white] (N27) circle [radius=0.15];% node[below right] {\Large $p_{t+s_1+s_2}$};

\draw [line width=0.05cm, fill=white] (Nt1) circle [radius=0.15];% node[below right] {\Large $p_{t+s_1+s_2+1}$}; 
\draw [line width=0.05cm, fill=white] (Nt3) circle [radius=0.15]; %node[below right] {\Large $p_{t+s_1+s_2+3}$}; 
\draw [line width=0.05cm, fill=white] (Nt5) circle [radius=0.15]; %node[below right] {\Large $p_{t+s_1+s_2+s_3-2}$}; 
\draw [line width=0.05cm, fill=white] (Ntt1) circle [radius=0.15]; %node[below right] {\Large $p_{t+s_1+s_2+s_3-1}$}; 
\draw [line width=0.05cm, fill=white] (Nt7) circle [radius=0.15];% node[below right] {\Large $p_{t+s_1+s_2+s_3}$};

\node at (3.5,0) {$$}; 
\draw [line width=0.015cm, decorate,decoration={brace,amplitude=10pt}](-0.25,4) -- (-0.25,6) %node[midway,xshift=-1.3cm,yshift=0.2cm] {\Large $s_1$}
node[midway,xshift=-0.8cm] {\Large $s_1$}; 
\draw [line width=0.015cm, decorate,decoration={brace,amplitude=10pt}](1.75,4) -- (1.75,6) node[midway,xshift=-0.8cm] {\Large $s_2$};
%node[black,midway,xshift=1.3cm,yshift=-0.2cm] {\Large vertices}; 
\draw [line width=0.015cm, decorate,decoration={brace,amplitude=10pt}](-0.25,0) -- (-0.25,2) node[black,midway,xshift=-0.8cm] {\Large $t_1$}; 
\draw [line width=0.015cm, decorate,decoration={brace,amplitude=10pt}](1.75,0) -- (1.75,2) node[black,midway,xshift=-0.8cm] {\Large $t_2$};

\end{tikzpicture} 
}}
\caption{The poset $\Pi_4$}
\label{poset4}
\end{minipage}
}}
\end{figure}

%%%%%%%%%%%%%%%%%%%%%

In \cite{N}, Gorenstein Hibi rings $\kk[\Pi]$ with $\Cl(\kk[\Pi]) \cong \ZZ$ or $\ZZ^2$ are discussed 
and the characterization of the associated posets is given. 
Note that $\kk[\Pi]$ is Gorenstein if and only if $\Pi$ is pure, i.e., all of the maximal chains in $\Pi$ have the same length (\cite{H87}). 
We can see that \cite[Example 3.1]{N} and the proof of \cite[Lemma 3.2]{N} works even for non-pure posets. 
Thus, we can characterize the Hibi rings $\kk[\Pi]$ with $\Cl(\kk[\Pi]) \cong \ZZ$ or $\ZZ^2$ as follows: 
\begin{prop}[{cf. \cite[Example 3.1 and Lemma 3.2]{N}}]\label{char:order}
Let $\Pi$ be a poset. Assume that $\kk[\Pi]$ is not a polynomial extension of a Hibi ring. 
\begin{itemize}
\item[(1)] If $\Cl(\kk[\Pi]) \cong \ZZ$, then $\calO_\Pi$ is isomorphic to $\calO_{\Pi_1(s_1,s_2)}$ for some $s_1,s_2$ with $d=s_1+s_2$. 
\item[(2)] If $\Cl(\kk[\Pi]) \cong \ZZ^2$, then $\calO_\Pi$ is isomorphic to $\calO_{\Pi_2(s_1,s_2,s_3,t)}$ for some $s_1,s_2,s_3,t$ with $d=s_1+s_2+s_3+t$, 
$\calO_{\Pi_3(s_1,s_2,t_1,t_2,t_3)}$ for some $s_1,s_2,t_1,t_2,t_3$ with $d=s_1+s_2+t_1+t_2+t_3$ or 
$\calO_{\Pi_4(s_1,s_2,t_1,t_2)}$ for some $s_1,s_2,t_1,t_2$ with $d=s_1+s_2+t_1+t_2$. 
\end{itemize}
\end{prop}

Given a poset $\Pi$, we define the \textit{comparability graph} of $\Pi$, denoted by $G(\Pi)$, 
as a graph on the vertex set $V(G(\Pi))=[d]$ with the edge set 
$$E(G(\Pi))=\{\{i,j\} : \text{$p_i$ and $p_j$ are comparable in $\Pi$}\}.$$ 
It is known that $G(\Pi)$ is perfect for any $\Pi$ (see e.g. \cite[Section 5.5]{Die}). 
Moreover, we see that $\calC_\Pi=\Stab_{G(\Pi)}$. 
\begin{prop}\label{prop:compare}
Let $\Pi$ be $\Pi_1(s_1,s_2)$ or $\Pi_2(s_1,s_2,s_3,t)$ or $\Pi_3(s_1,s_2,t_1,t_2,t_3)$. 
Then $\calO_\Pi$ is unimodularly equivalent to $\calC_\Pi=\Stab_{G(\Pi)}$. 
\end{prop}
\begin{proof}
This directly follows from Theorem~\ref{X}. 
\end{proof}

\bigskip

%%%%%%%%%%%%%%%%%%%%%%%%%%%%%%%%%%%%%%%%%%%%%%%%%%%%%%%%%%%%%%%%%%%%%%%%%%%%%%%%%%%%%%%%%%%%%%%%%%%%%%%%%%%%%%%%%%%%%%%%%%%%%%%%%%%%
%%%%%%%%%%%%%%%%%%%%%%%%%%%%%%%%%%%%%%%%%%%%%%%%%%%%%%%%%%%%%%%%%%%%%%%%%%%%%%%%%%%%%%%%%%%%%%%%%%%%%%%%%%%%%%%%%%%%%%%%%%%%%%%%%%%%
%%%%%%%%%%%%%%%%%%%%%%%%%%%%%%%%%%%%%%%%%%%%%%%%%%%%%%%%%%%%%%%%%%%%%%%%%%%%%%%%%%%%%%%%%%%%%%%%%%%%%%%%%%%%%%%%%%%%%%%%%%%%%%%%%%%%
%%%%%%%%%%%%%%%%%%%%%%%%%%%%%%%%%%%%%%%%%%%%%%%%%%%%%%%%%%%%%%%%%%%%%%%%%%%%%%%%%%%%%%%%%%%%%%%%%%%%%%%%%%%%%%%%%%%%%%%%%%%%%%%%%%%%
%%%%%%%%%%%%%%%%%%%%%%%%%%%%%%%%%%%%%%%%%%%%%%%%%%%%%%%%%%%%%%%%%%%%%%%%%%%%%%%%%%%%%%%%%%%%%%%%%%%%%%%%%%%%%%%%%%%%%%%%%%%%%%%%%%%%
\subsection{Stable set rings with small class groups}

For stable set rings, if their class groups are isomorphic $\ZZ$ or $\ZZ^2$, 
then we see that we can associate Hibi rings as follows: 
\begin{thm}\label{thm:stab}
Let $G$ be a perfect graph. 
\begin{itemize}
\item[(1)] Assume that $\Cl(\kk[\Stab_G]) \cong \ZZ$. 
Then $\Stab_G$ is unimodularly equivalent to $\calO_{\Pi_1(s_1,s_2)}$ for some $s_1,s_2 \in \ZZ_{>0}$. 
\item[(2)] Assume that $\Cl(\kk[\Stab_G]) \cong \ZZ^2$.
Then $\Stab_G$ is unimodularly equivalent to $\calO_{\Pi_2(s_1,s_2,s_3,t)}$ for some $s_1,s_2,s_3 \in \ZZ_{>0}$ and $t \in \ZZ_{\geq 0}$, 
or $\calO_{\Pi_3(s_1,s_2,t_1,t_2,t_3)}$ for some $s_1,s_2 \in \ZZ_{>0}$ and $t_1,t_2,t_3 \in \ZZ_{\geq 0}$, 
\end{itemize}
\end{thm}
\begin{proof}
Let $Q_0,Q_1,\ldots,Q_n$ be the maximal cliques of $G$. Then $\Cl(\kk[\Stab_G]) \cong \ZZ^n$ by Proposition~\ref{prop:class_stab}. 
If $v\in \bigcap_{i=0}^n Q_i \ne \emptyset$, then $v$ is adjacent to any vertex in $G$, so we see that $\Stab_G=\Pyr(\Stab_{G \setminus v})$. 
In particular, $\kk[\Stab_G]\cong \kk[\Stab_{G\setminus v}][x_v]$. Thus, we may assume that $\bigcap_{i=0}^n Q_i = \emptyset$. 

Let $n=1$. 
We can see that $G=G(\Pi_1(s_1,s_2))$, where $s_1=|Q_0|$ and $s_2=|Q_1|$, 
by observing \eqref{facets:stab} for $G(\Pi_1(s_1,s_2))$ and the definition of $\calC_{\Pi_1(s_1,s_2)}$. 
It follows from Theorem~\ref{X} that $\kk[\calO(\Pi_1(s_1,s_2))] \cong \kk[\calC(\Pi_1(s_1,s_2))]=\kk[\Stab(G(\Pi_1(s_1,s_2)))]$.

Let $n=2$. 
\begin{itemize}
\item[(i)] If $Q_0\cap Q_1=Q_0\cap Q_2=Q_1\cap Q_2=\emptyset$, then we can see that $G=G(\Pi_2(s_1,s_2,s_3,0))$, 
where $s_1=|Q_0|$, $s_2=|Q_1|$ and $s_3=|Q_2|$. 
\item[(ii)] If $Q_0\cap Q_1=Q_0\cap Q_2 = \emptyset$ and $Q_1\cap Q_2 \ne \emptyset$, 
then we can see that $G=G(\Pi_2(s_1,s_2,s_3,t))$, where $s_1=|Q_1\setminus Q_2|$, $s_2=|Q_2\setminus Q_1|$, $s_3=|Q_0|$ and $t=|Q_1\cap Q_2|$.
\item[(iii)]  If $Q_0\cap Q_1, Q_0\cap Q_2 \ne \emptyset$ and $Q_1\cap Q_2=\emptyset$, 
then we can see that $G=G(\Pi_3(s_1,s_2,t_1,t_2,t_3))$, where $s_1=|Q_1\setminus Q_0|$, $s_2=|Q_2\setminus Q_0|$, 
$t_1=|Q_0\cap Q_1|$, $t_2=|Q_0\cap Q_2|$ and $t_3=|Q_0\setminus (Q_1\cup Q_2)|$.
\item[(iv)]  If $Q_0\cap Q_1, Q_0\cap Q_2, Q_1\cap Q_2\ne \emptyset$, 
then we see that $Q=(Q_0\cap Q_1)\cup (Q_0\cap Q_2)\cup (Q_1\cap Q_2)$ is also a maximal clique which is different from $Q_0,Q_1,Q_2$. 
This contradicts to $\Cl(\kk[\Stab_G])\cong \ZZ^2$ by Proposition~\ref{prop:class_stab}.
\end{itemize}
%It follows from Theorem~\ref{X} that $\kk[\calO_\Pi] \cong \kk[\calC_\Pi]=\kk[\Stab_{G(\Pi)}]$ for each posets $\Pi=\Pi_1$, $\Pi_2$ and $\Pi_3$. 
\end{proof}

\bigskip

%%%%%%%%%%%%%%%%%%%%%%%%%%%%%%%%%%%%%%%%%%%%%%%%%%%%%%%%%%%%%%%%%%%%%%%%%%%%%%%%%%%%%%%%%%%%%%%%%%%%%%%%%%%%%%%%%%%%%%%%%%%%%%%%%%%%
%%%%%%%%%%%%%%%%%%%%%%%%%%%%%%%%%%%%%%%%%%%%%%%%%%%%%%%%%%%%%%%%%%%%%%%%%%%%%%%%%%%%%%%%%%%%%%%%%%%%%%%%%%%%%%%%%%%%%%%%%%%%%%%%%%%%
%%%%%%%%%%%%%%%%%%%%%%%%%%%%%%%%%%%%%%%%%%%%%%%%%%%%%%%%%%%%%%%%%%%%%%%%%%%%%%%%%%%%%%%%%%%%%%%%%%%%%%%%%%%%%%%%%%%%%%%%%%%%%%%%%%%%
%%%%%%%%%%%%%%%%%%%%%%%%%%%%%%%%%%%%%%%%%%%%%%%%%%%%%%%%%%%%%%%%%%%%%%%%%%%%%%%%%%%%%%%%%%%%%%%%%%%%%%%%%%%%%%%%%%%%%%%%%%%%%%%%%%%%
%%%%%%%%%%%%%%%%%%%%%%%%%%%%%%%%%%%%%%%%%%%%%%%%%%%%%%%%%%%%%%%%%%%%%%%%%%%%%%%%%%%%%%%%%%%%%%%%%%%%%%%%%%%%%%%%%%%%%%%%%%%%%%%%%%%%
\subsection{Edge rings with small class groups}
The goal of this subsection is to give a complete description of $G$ satisfying the odd cycle condition 
with $\Cl(\kk[G]) \cong \ZZ$ or $\ZZ^2$. 
Throughout this subsection, we let $G$ be a connected graph satisfying the odd cycle condition. 
We discuss $G$ by dividing it into whether $G$ is bipartite or not. 
\begin{prop}\label{rank}
Let $\Cl(\kk[G]) \cong \ZZ^t$. If $G$ contains at least two non-bipartite blocks, then $t\ge 4$.
\end{prop}
\begin{proof} 
Let $B_1,\ldots,B_m$ be the blocks of $G$, where $m \geq 2$, and assume that at least two of them are non-bipartite. 
We prove the assertion by induction on $m$.

Let $m = 2$. Then $B_1$ and $B_2$ are non-bipartite. 
Thus, $B_1$ and $B_2$ have primitive odd cycle $C_1=p_0 \cdots p_{2k_1} p_0$ and $C_2=q_0 \cdots q_{2k_2} q_0$ ($1\le k_1 \le k_2$), respectively. 
Let $v \in V(B_1) \cap V(B_2)$ be a unique cut vertex. Then we see that every vertex in $V(G) \setminus \{v\}$ is regular, 
implying that $|\Psi_r|\ge |V(G)|-1=d-1$ and $G$ has $|\Psi_f| \ge \min \{|V(C_1),V(C_2)|\}=2k_1+1$ by Lemma~\ref{fundamental} (2). 
%\begin{itemize}
%\item If $v\notin V(C_1) \cup V(C_2)$, then $C_1$ and $C_2$ have no common vertex and paths connecting them have at least two length; a contradiction to odd cycle condition,
%\item If $v\in V(C_1)$ and $v\notin V(C_2)$ (we may assume that $v=p_0$), we can take two fundamental sets on $G$: Let $S_1=\{p_1,p_3,\ldots,p_{2k_1-1}\}$ and $S_2=\{p_2,p_4,\ldots,p_{2k_1}\}$. Then there exists a fundamental set $T_i$ for $i=1,2$ such that $S_i \subset T_i$ and $V(B(T_i))=V(B_1)$ by same discussions in Lemma~\ref{fundamental} (1), that is, we can give two or more extra fundamental sets.
%\item If $v\in V(C_1), V(C_2)$ (we may assume that $v=p_0=q_0$), we also can take extra two or more fundamental sets on $G$: Let $U_1=\{q_1,q_3,\ldots,q_{2k_2-1}\}$ and $U_2=\{q_2,q_4,\ldots,q_{2k_2}\}$ and take $S_1$ and $S_2$ above. Then there exists a fundamental set $T'_{i,j}$ for $i=1,2$ and $j=1,2$ such that $S_i\cup U_j \subset T'_{i,j}$ and $V(B(T'_{i,j}))=V(G)$ by Lemma~\ref{fundamental} (1).
%\end{itemize} 
\begin{itemize}
\item Suppose that $v\notin V(C_1) \cup V(C_2)$. Then there is a path containing $v$ which connects $V(C_1)$ and $V(C_2)$. 
This is a contradiction to what $G$ satisfies the odd cycle condition. 
\item Suppose that $v \in V(C_1) \setminus V(C_2)$. Let, say, $v=p_0$. Then we can take two fundamental sets on $G$ as follows. 
Let $S_1=\{p_1,p_3,\ldots,p_{2k_1-1}\}$ and $S_2=\{p_2,p_4,\ldots,p_{2k_1}\}$. 
Then there exist fundamental sets $T_1$ and $T_2$ such that $S_i \subset T_i$ and $V(B(T_i))=V(B_1)$ for $i=1,2$ by Lemma~\ref{fundamental} (1). 
Namely, we can get two (or more) fundamental sets. Hence, 
$$t=|\Psi|-\dim \kk[G]=|\Psi_f|+|\Psi_r|-d \ge (2k_1+1)+2+(d-1)-d\ge 4.$$ 
\item Suppose that $v\in V(C_1) \cap V(C_2)$. Let, say, $v=p_0=q_0$. 
Then we can also take two (or more) fundamental sets on $G$ as follows. 
Let $U_1=\{q_1,q_3,\ldots,q_{2k_2-1}\}$ and $U_2=\{q_2,q_4,\ldots,q_{2k_2}\}$ and take $S_1$ and $S_2$ above. 
Then there exist fundamental sets $T'_{i,j}$ for $i=1,2$ and $j=1,2$ such that 
$S_i\cup U_j \subset T'_{i,j}$ and $V(B(T'_{i,j}))=V(G)$ by Lemma~\ref{fundamental} (1). 
Hence, as above, we obtain that $t \geq 4$. 
\end{itemize} 

\medskip

Suppose that $m\ge 3$. Take a block $B_i$ whose degree is $1$ on $\Block(G)$. 
Then $B_i$ has a unique cut vertex $u$ on $G$. Let $H=G\setminus (V(B_i)\setminus \{u\})$ and $b=|V(B_i)|$. 
Note that $H$ has an odd cycle by assumption and every vertex in $B_i\setminus u$ is regular on $G$. Thus, we have 
\begin{align*}
|\Psi_r(G)|=\begin{cases}
|\Psi_r(H)|+(b-1), &\text{if (i) $u$ is non-regular in $H$ and in $G$}, \\
|\Psi_r(H)|+(b-1)-1, &\text{if (ii) $u$ is regular in $H$ and non-regular in $G$}.
\end{cases}
\end{align*}
Notice that if $u$ is regular in $H$ and $G$, then $B_i \setminus u$ and all connected components of $H\setminus u$ contain an odd cycle, 
a contradiction by the same reason as discussed above. 
Moreover, it never happens that $u$ is non-regular on $H$ and regular on $G$. 

In the case of (ii), we have $|\Psi_f(G)|\ge |\Psi_f(H)|+1$ by Remark~\ref{remark}. 
Therefore, in the case of (i), we obtain by inductive hypothesis the following: 
\begin{align*}
t&=|\Psi_r(G)|+|\Psi_f(G)|-d \ge (|\Psi_r(H)|+(b-1)-1)+(|\Psi_f(H)|+1)-d \\
&=|\Psi_r(H)|+|\Psi_f(H)|-(d-(b-1)) =|\Psi(H)|-\dim \kk[H] \\
&\ge 4. 
\end{align*}
\end{proof}

%Let $K_{r,s}$ be the complete bipartite graph on $r+s$ vertices with the partition consisting of $r$ vertices and $s$ vertices. 
\begin{lem}\label{acceptable}
Let $G$ be a bipartite graph with the partition $V(G)=V_1\sqcup V_2$. 
If $G$ is not a complete bipartite graph, then there exists an acceptable set contained in s$V_1$. 
\end{lem}
\begin{proof}
Let $n_1=|V_1|$ and $n_2=|V_2|$. Note that $n_1,n_2\ge 2$ since $G$ is connected and non-complete bipartite. 
Take a vertex $v_0\in V_1$ such that $\deg(v_0)=\min \{\deg(v) : v\in V_1\}$. Then $\deg(v_0) < n_2$. 
Moreover, $G\setminus V(B(\{v_0\}))$ contains connected components $C_1,\ldots,C_n$ which have at least one edge, 
and it might have some isolated vertices $v_1,\ldots,v_m$ in $V_1$. For $i \in [n]$, 
let $A_i=\{v_0,v_1,\ldots,v_m\}\cup \bigl( \bigcup_{j\in [n],j\ne i} V(C_j)\cap V_1\bigr)$. Then each $A_i$ is acceptable. 
In fact, $B(A_i)$ is connected since $G$ is connected, and $G\setminus V(B(A_i))=C_i$ is a connected graph with at least one edge. 
\end{proof}

We define two graphs $K_{s_1,s_2}^{t_1,t_2}$ and $K_{1,s_1,s_2}^{t_1,t_2}$ as follows: 
\begin{defi}\label{bipartite}
Let $s_1$, $s_2$, $t_1$, $t_2$ be integers with $0 \leq t_1 <s_1$ and $0 \leq t_2 < s_2$. 
\begin{itemize}
\item Let $K_{s_1,s_2}^{t_1,t_2}$ denote the bipartite graph on the vertex set $V(K_{s_1,s_2}^{t_1,t_2})=[d]$ ($d=s_1+s_2+t_1+t_2$) 
with the edge set \begin{align*}
E(K_{s_1,s_2}^{t_1,t_2}) &=\{\{i,j\} : 1\le i \le s_1+t_1, s_1+t_1+t_2+1 \le j \le d\} \\
&\cup \{\{i,j\} : 1 \le i \le s_1, s_1+t_1+1 \le j \le d \}.\end{align*}
See Figure~\ref{graph2}. 
\item Let $K_{1,s_1,s_2}^{t_1,t_2}$ denote the graph on the vertex set $V(K_{1,s_1,s_2}^{t_1,t_2})=[d+1]$ ($d=s_1+s_2+t_1+t_2$) 
with the edge set $$E(K_{1,s_1,s_2}^{t_1,t_2})=E(K_{s_1,s_2}^{t_1,t_2})\cup \{\{i,d+1\} : 1\le i \le s_1 \text{ or } s_1+t_1+t_2+1 \le i \le d\}.$$ 
See Figure~\ref{graph22}. 
\end{itemize}
\end{defi}
%%%%%%%%%%%%%%%%%%%%%%%
\begin{figure}[h]%3
{\scalebox{0.8}{
\begin{minipage}{0.49\columnwidth}
\centering
{\scalebox{0.65}{
\begin{tikzpicture}[line width=0.02cm]

\coordinate (N11) at (0,0); 
\coordinate (N12) at (0,1); 
\coordinate (N13) at (0,2); 
\coordinate (N14) at (0,3);
\coordinate (N15) at (0,4); 
\coordinate (N16) at (0,5); 
\coordinate (N17) at (0,6); 
\coordinate (N18) at (0,7);

\coordinate (N21) at (4,0); 
\coordinate (N22) at (4,1); 
\coordinate (N23) at (4,2); 
\coordinate (N24) at (4,3);
\coordinate (N25) at (4,4); 
\coordinate (N26) at (4,5); 
\coordinate (N27) at (4,6); 
\coordinate (N28) at (4,7);

\coordinate (Ntt1) at (1,2.5);

%edge
 
\draw[dotted]  (0,1.8)--(0,2.2);
\draw[dotted]  (0,4.8)--(0,5.2);

\draw[dotted]  (4,1.8)--(4,2.2);
\draw[dotted]  (4,4.8)--(4,5.2);

\draw  (N18)--(N21);
\draw  (N18)--(N22);
\draw  (N18)--(N24);
\draw  (N18)--(N25);
\draw  (N18)--(N27);
\draw  (N18)--(N28);

\draw  (N17)--(N21);
\draw  (N17)--(N22);
\draw  (N17)--(N24);
\draw  (N17)--(N25);
\draw  (N17)--(N27);
\draw  (N17)--(N28);

\draw  (N15)--(N21);
\draw  (N15)--(N22);
\draw  (N15)--(N24);
\draw  (N15)--(N25);
\draw  (N15)--(N27);
\draw  (N15)--(N28);

\draw  (N14)--(N25);
\draw  (N14)--(N27);
\draw  (N14)--(N28);

\draw  (N12)--(N25);
\draw  (N12)--(N27);
\draw  (N12)--(N28);

\draw  (N11)--(N25);
\draw  (N11)--(N27);
\draw  (N11)--(N28);
%node

\draw [line width=0.05cm, fill=white] (N11) circle [radius=0.15] node[] at (-1,0) {\Large $s_1+t_1$}; 
\draw [line width=0.05cm, fill=white] (N12) circle [radius=0.15] node[] at (-1.5,1) {\Large $s_1+t_1-1$};
\draw [line width=0.05cm, fill=white] (N14) circle [radius=0.15] node[] at (-1,3) {\Large $s_1+1$};
\draw [line width=0.05cm, fill=white] (N15) circle [radius=0.15] node[] at (-0.5,4) {\Large $s_1$};
\draw [line width=0.05cm, fill=white] (N17) circle [radius=0.15] node[] at (-0.5,6) {\Large $2$};
\draw [line width=0.05cm, fill=white] (N18) circle [radius=0.15] node[] at (-0.5,7) {\Large $1$};

\draw [line width=0.05cm, fill=white] (N21) circle [radius=0.15] node[] at (5.5,0) {\Large $s_1+t_1+1$};
\draw [line width=0.05cm, fill=white] (N22) circle [radius=0.15] node[] at (5.5,1) {\Large $s_1+t_1+2$};
\draw [line width=0.05cm, fill=white] (N24) circle [radius=0.15] node[] at (5.5,3) {\Large $s_1+t_1+t_2$};
\draw [line width=0.05cm, fill=white] (N25) circle [radius=0.15] 
node[] at (5,4.3) {\Large $s_1+t_1$} node[] at (5.5,3.8) {\Large $+t_2+1$};
\draw [line width=0.05cm, fill=white] (N27) circle [radius=0.15] node[] at (4.8,6) {\Large $d-1$};
\draw [line width=0.05cm, fill=white] (N28) circle [radius=0.15] node[] at (4.5,7) {\Large $d$};

%\draw [line width=0.05cm, fill=white] (Ntt1) circle [radius=0.15] ; 

\node at (7,0) {$$}; 

\end{tikzpicture}
} }
\caption{\; \\ The graph $K_{s_1,s_2}^{t_1,t_2}$}
\label{graph2}
\end{minipage}
\begin{minipage}{0.49\columnwidth}
\centering
{\scalebox{0.6}{
\begin{tikzpicture}[line width=0.02cm]

\coordinate (N11) at (0,0); 
\coordinate (N12) at (0,1); 
\coordinate (N13) at (0,2); 
\coordinate (N14) at (0,3);
\coordinate (N15) at (0,4); 
\coordinate (N16) at (0,5); 
\coordinate (N17) at (0,6); 
\coordinate (N18) at (0,7);

\coordinate (N21) at (4,0); 
\coordinate (N22) at (4,1); 
\coordinate (N23) at (4,2); 
\coordinate (N24) at (4,3);
\coordinate (N25) at (4,4); 
\coordinate (N26) at (4,5); 
\coordinate (N27) at (4,6); 
\coordinate (N28) at (4,7);

\coordinate (Ntt1) at (2,8);

%edge
 
\draw[dotted]  (0,1.8)--(0,2.2);
\draw[dotted]  (0,4.8)--(0,5.2);

\draw[dotted]  (4,1.8)--(4,2.2);
\draw[dotted]  (4,4.8)--(4,5.2);

\draw  (N18)--(N21);
\draw  (N18)--(N22);
\draw  (N18)--(N24);
\draw  (N18)--(N25);
\draw  (N18)--(N27);
\draw  (N18)--(N28);

\draw  (N17)--(N21);
\draw  (N17)--(N22);
\draw  (N17)--(N24);
\draw  (N17)--(N25);
\draw  (N17)--(N27);
\draw  (N17)--(N28);

\draw  (N15)--(N21);
\draw  (N15)--(N22);
\draw  (N15)--(N24);
\draw  (N15)--(N25);
\draw  (N15)--(N27);
\draw  (N15)--(N28);

\draw  (N14)--(N25);
\draw  (N14)--(N27);
\draw  (N14)--(N28);

\draw  (N12)--(N25);
\draw  (N12)--(N27);
\draw  (N12)--(N28);

\draw  (N11)--(N25);
\draw  (N11)--(N27);
\draw  (N11)--(N28);

\draw  (Ntt1)--(N18);
\draw  (Ntt1)--(N17);
\draw  (Ntt1)--(N15);
\draw  (Ntt1)--(N28);
\draw  (Ntt1)--(N27);
\draw  (Ntt1)--(N25);
%node

\draw [line width=0.05cm, fill=white] (N11) circle [radius=0.15] node[] at (-1,0) {\Large $s_1+t_1$}; 
\draw [line width=0.05cm, fill=white] (N12) circle [radius=0.15] node[] at (-1.5,1) {\Large $s_1+t_1-1$};
\draw [line width=0.05cm, fill=white] (N14) circle [radius=0.15] node[] at (-1,3) {\Large $s_1+1$};
\draw [line width=0.05cm, fill=white] (N15) circle [radius=0.15] node[] at (-0.5,4) {\Large $s_1$};
\draw [line width=0.05cm, fill=white] (N17) circle [radius=0.15] node[] at (-0.5,6) {\Large $2$};
\draw [line width=0.05cm, fill=white] (N18) circle [radius=0.15] node[] at (-0.5,7) {\Large $1$};

\draw [line width=0.05cm, fill=white] (N21) circle [radius=0.15] node[] at (5.5,0) {\Large $s_1+t_1+1$};
\draw [line width=0.05cm, fill=white] (N22) circle [radius=0.15] node[] at (5.5,1) {\Large $s_1+t_1+2$};
\draw [line width=0.05cm, fill=white] (N24) circle [radius=0.15] node[] at (5.5,3) {\Large $s_1+t_1+t_2$};
\draw [line width=0.05cm, fill=white] (N25) circle [radius=0.15] 
node[] at (5,4.3) {\Large $s_1+t_1$} node[] at (5.5,3.8) {\Large $+t_2+1$};
\draw [line width=0.05cm, fill=white] (N27) circle [radius=0.15] node[] at (4.8,6) {\Large $d-1$};
\draw [line width=0.05cm, fill=white] (N28) circle [radius=0.15] node[] at (4.5,7) {\Large $d$};

\draw [line width=0.05cm, fill=white] (Ntt1) circle [radius=0.15] node[] at (2.8,8.2) {\Large $d+1$};

\node at (7,0) {$$}; 

\end{tikzpicture}
} }
\caption{\; \\ The graph $K_{1,s_1,s_2}^{t_1,t_2}$}
\label{graph22}
\end{minipage}
}}
\end{figure}

%%%%%%%%%%%%%%%%%%%%%%%
Note that $K_{s_1,s_2}^{t_1,t_2}$ (resp. $K_{1,s_1,s_2}^{t_1,t_2}$) is 
a complete bipartite graph $K_{s_1,s_2}$ (resp. a complete $3$-partite graph $K_{1,s_1,s_2}$) minus the edges of $K_{t_1,t_2}$. 
Thus, $K_{s_1,s_2}^{t_1,t_2}$ is bipartite, but $K_{1,s_1,s_2}^{t_1,t_2}$ is not. 
When $t_1=t_2=0$, we regard $K_{s_1,s_2}^{t_1,t_2}$ (resp. $K_{1,s_1,s_2}^{t_1,t_2}$) as $K_{s_1,s_2}$ (resp. $K_{1,s_1,s_2}$) itself.

\medskip

First, we discuss the case of bipartite graphs. 
We give the characterization of which $\Cl(\kk[G])$ is isomorphic to $\ZZ$ or $\ZZ^2$ in terms of $G$ for bipartite graphs. 
By Proposition~\ref{product}, we may assume that $G$ is $2$-connected.

\begin{thm}\label{thm:bip_class_group}
Let $G$ be a $2$-connected bipartite graph with its partition $V(G)=V_1\sqcup V_2$.
\begin{itemize}
\item[(1)] $\Cl(\kk[G]) \cong \ZZ$ if and only if $G$ is a complete bipartite graph $K_{s_1,s_2}$ with $s_1,s_2\ge 2$.
\item[(2)] $\Cl(\kk[G]) \cong \ZZ^2$ if and only if $G$ is a bipartite graph $K_{s_1,s_2}^{t_1,t_2}$ for some $t_1,t_2\ge 1$ and $s_1,s_2 \ge 2$.
\end{itemize}
\end{thm}
\begin{proof}
(1) Since every vertex in $G$ is ordinary, 
we see that $\rank(\Cl(\kk[G]))=|\Psi|-\dim \kk[G]=|\Psi_o|+|\Psi_a|-(d-1)=|\Psi_a|+1$ (see Theorem~\ref{thm:class_edge}). 
If $G$ is not a complete bipartite, then $G$ contains an acceptable set by Lemma~\ref{acceptable} and we have $t\ge 2$. 
Therefore, we can see that $G$ is a complete bipartite and $s_1,s_2\ge 2$ since $G$ is $2$-connected. 
Conversely, if $G$ is a complete bipartite graph $K_{s_1,s_2}$ with $s_1,s_2\ge 2$, 
then it is easy to check that $\Cl(K_{s_1,s_2}) \cong \ZZ$. 

\noindent
(2) Assume that $\Cl(\kk[G]) \cong \ZZ^2$. 
By (1), $G$ cannot be a complete bipartite graph. Thus, we can take $v_0,v_1,\ldots,v_m$, $C_1,\ldots,C_n$ and $A_1,\ldots,A_n$ mentioned in Lemma~\ref{acceptable}. 
We can see that $n=1$ since $t=|\Psi_a|+1=2$. Moreover, we see that $B(\{v_0,v_1\ldots,v_m\})$ is a complete bipartite by definition of $v_0,v_1,\ldots,v_m$. 
Note that $A_1=\{v_0,v_1,\ldots,v_m\}$. Thus, it is enough to show that $C_1$ and $G_W$ are complete bipartite graphs, 
where $W=(V(C_1)\cap V_1)\cup N_G(v_0)$. 

If $C_1$ is not a complete bipartite graph, then we can take an acceptable set $A\subset V_1$ of $C_1$ by Lemma~\ref{acceptable} 
and $A'$ is an acceptable set of $G$, where 
\begin{align*}
A'=\begin{cases}
A  &\text{ if } N_G(A)\cap N_G(v_0)=\emptyset, \\
A\cup A_1  &\text{ if } N_G(A)\cap N_G(v_0)\ne \emptyset, 
\end{cases}
\end{align*}
a contradiction. Similarly, if $G_W$ is not a complete bipartite graph, then we can take an acceptable set of $G$ by the same way in Lemma~\ref{acceptable}. 
Let $s_1=|V(C_1) \cap V_1|$, $s_2=|N_G(A_1)|$, $t_1=|A_1|$ and $t_2=|V(C_1)\cap V_2|$. 
Then $G$ coincide with $K_{s_1,s_2}^{t_1,t_2}$ and we see that $s_1,s_2\ge 2$ since $G$ is $2$-connected. 
Conversely, the subset $\{s_1+1,\ldots,s_1+t_1\}$ of $V(K_{s_1,s_2}^{t_1,t_2})$ is a unique acceptable set of $K_{s_1,s_2}^{t_1,t_2}$ 
and we have $\Cl(\kk[K_{s_1,s_2}^{t_1,t_2}]) \cong \ZZ^2$. 
\end{proof}

\medskip

Next, we discuss non-bipartite graphs. 
\begin{lem}\label{lem:regular}
Let $G$ be a $2$-connected graph with primitive odd cycles $C_i=p_{i,0} \cdots p_{i,2k_i} p_{i,0}$ for $i\in [m]$, where $1\le k_1 \le \cdots \le k_m$, 
and let $P=x_0 x_1 \cdots x_l$ with $l\ge 2$ be a primitive path whose end vertices $x_0,x_l$ are in $V(C_m)$ and $x_k \notin V(C_m)$ for all $k\in [l-1]$. 
\begin{itemize}
\item[(1)] For $j\in \{0,1,\ldots,2k_m\}$, $p_{m,j}$ is non-regular in $G$ if and only if $p_{m,j}\in V(C_i)$ for all $i \in [m]$.
\item[(2)] Suppose that $x_0=p_{m,0}$ and $x_l=p_{m,j}$ ($j\ne 1,2k_m$). Then $C_m$ has a regular vertex in $G$.
\item[(3)] Suppose that $\{x_0,x_l\}=\{p_{m,j},p_{m,j+1}\}$ for $j\in \{0,1,\ldots,2k_m\}$, where $p_{2k_m+1}=p_0$ and $l=2l'+1$. 
Then there are two different fundamental sets $T_1, T_2$ such that 
$E(C_m)\setminus \{p_{m,j},p_{m,j+1}\}\subset E(B(T_i))$ and $\{p_{m,j},p_{m,j+1}\} \notin E(B(T_i))$ for $i=1,2$. 
\end{itemize}
\end{lem}
\begin{proof}
(1) If there exists $i\in [m]$ such that $p_{m,j}\notin V(C_i)$, 
then the connected graph $G\setminus p_{m,j}$ contains $C_i$ as a subgraph. Hence, $p_{m,j}$ is regular in $G$. 
Conversely, if $p_{m,j}\in V(C_i)$ for all $i \in [m]$, then the connected graph $G\setminus p_{m,j}$ has no odd cycles. 
Thus, $p_{m,j}$ is non-regular. 

\noindent
(2) Let $C=x_0 x_1\cdots x_l p_{m,j-1} p_{m,j-2} \cdots p_{m,0}$ and 
$C'=x_0\cdots x_l p_{m,j+1} p_{m,j+2} \cdots p_{m,2k_m} p_{m,0}$. 
Then $C$ or $C'$ is a primitive odd cycle because $C_m$ is a primitive odd cycle. Therefore, $p_{m,1},\ldots,p_{m,j-1}$ or $p_{m,j+1},\ldots,p_{m,2k_m}$ are regular vertices in $V(C_m)$.

\noindent
(3) We may assume that $j=0$. Let $S_1=\{p_{m,2},p_{m,4},\ldots, p_{m,2k_m},x_1,x_3,\cdots,x_{2l'-1}\}$ and 
$S_2=\{p_{m,2},p_{m,4},\ldots, p_{m,2k_m},x_2,x_4,\cdots,x_{2l'}\}$ are independent sets and $N_G(S_i)$ is connected for $i=1,2$. 
Therefore, the statement immediately follows from Lemma~\ref{fundamental} (1). 
\end{proof}

%For a non-bipartite graph $G$, we also give the characterization of which $\kk[G ]$ is isomorphic to $\ZZ$ or $\ZZ^2$ in terms of $G$. 
%By Propositions~\ref{product} and \ref{rank}, we may assume that $G$ is 2-connected. 

\begin{thm}\label{thm:nonbip_class_group}
Let $G$ be a $2$-connected non-bipartite graph. 
\begin{itemize}
\item[(1)] $\Cl(\kk[G]) \cong \ZZ$ if and only if $G$ is obtained by one of the following two ways. 

For the complete bipartite graph $K_{s_1,s_2}$ with $s_1,s_2\ge 2$, %with the partition $V(K_{s_1,s_2})=\{1,\ldots,s_1\}\sqcup \{s_1+1,\ldots,s_1+s_2\}$, $s_1,s_2\ge 2$,
\begin{itemize}
\item[(1-1)] choose $i$ and $j$ from the different partition, respectively, and connect them by a path of even length at least $2$ (see {\em Figure~\ref{graph(a1)}}); or 
\item[(1-2)] choose $i$ and $j$ from the same partition and connect them by a path of odd length (see {\em Figure~\ref{graph3}}). 
\end{itemize} 

\medskip

\item[(2)] $\Cl(\kk[G]) \cong \ZZ^2$ if and only if $G$ is obtained by one of the following six ways. 

For the complete bipartite graph $K_{s_1,s_2}$ and $K_{t_1,t_2}$ with $s_1,s_2,t_1,t_2\ge 2$; 
\begin{itemize}
\item[(2-1)] choose $i$ and $j$ (resp., $k$ and $l$) from the different partition of $K_{s_1,s_2}$ (resp., $K_{t_1,t_2}$), respectively, 
and connect $i$ and $k$ by a path $P_{i,k}$, $j$ and $l$ by a path $P_{j,l}$ such that the sum of the lengths of $P_{i,k}$ and $P_{j,l}$ is odd (see {\em Figure~\ref{graph(b1)}}); or 
\item[(2-2)] choose $i$ and $j$ from the same partition of $K_{s_1,s_2}$ and choose $k$ and $l$ from the different partition of $K_{t_1,t_2}$, respectively, 
and connect $i$ and $k$ by a path $P_{i,k}$, $j$ and $l$ by a path $P_{j,l}$ such that the sum of the lengths of $P_{i,k}$ and $P_{j,l}$ is even (see {\em Figure~\ref{graph(b2)}}); or 
\item[(2-3)] choose $i$ and $j$ (resp., $k$ and $l$) from the same partition of $K_{s_1,s_2}$ (resp., $K_{t_1,t_2}$), respectively, 
and connect $i$ and $k$ by a path $P_{i,k}$, $j$ and $l$ by a path $P_{j,l}$ such that the sum of the lengths of $P_{i,k}$ and $P_{j,l}$ is odd (see {\em Figure~\ref{graph(b3)}}); 
\end{itemize}
where if the length of the path is allowed to be $0$, then identify $i$ and $k$ (or $j$ and $l$). 

For the bipartite graph $K_{s_1,s_2}^{t_1,t_2}$ with $s_1,s_2\ge 2$; 
\begin{itemize}
\item[(2-4)] choose $i$ and $j$ from the different partition, respectively, and connect them by a path of even length at least $2$ (see {\em Figure~\ref{graph(b4)}}); or 
\item[(2-5)] choose $i$ and $j$ from the same partition and connect them by a path of odd length (see {\em Figure~\ref{graph(b5)}}); or 
\item[(2-6)] $G$ coincides with $K_{1,s_1,s_2}^{t_1,t_2}$ with $s_1,s_2\ge 2$ (see {\em Figure~\ref{graph22}}). 
\end{itemize}
\end{itemize}
\end{thm}
%
%
%%%%%%
\input{figure_edge2}
%%%%%%
%
%
\begin{rem}
Regarding the above constructions, although those graphs are not bipartite due to the additional paths appearing in each case of (1-1),(1-2) and (2-1)---(2-5), 
we observe that every odd cycle in each graph passes through those additional paths. 
Namely, if $C$ and $C'$ are odd cycles in a given graph as above, then $C$ and $C'$ always share the additional paths. 

On the other hand, it is well-known that the toric ideal of $\kk[G]$ is generated by the binomials corresponding to primitive even closed walks appearing in $G$. 
See, e.g. \cite[Section 5.3]{HHO}, for the details. 

Hence, for the graphs $G$ constructed like Theorem~\ref{thm:nonbip_class_group}, 
we see that the variables corresponding to the edges of the additional paths never appear in generators of the toric ideal of $G$. 
This means that $\kk[G]$ is isomorphic to the polynomial extension of $\kk[G']$, where $G'$ is the graph obtained by removing all the edges in the additional paths, 
i.e., $G'$ is $K_{s_1,s_2}$ or two copies of $K_{s_1,s_2}$ or $K_{s_1,s_2}^{t_1,t_2}$ by construction. 
\end{rem}
\begin{proof}[Proof of Theorem~\ref{thm:nonbip_class_group}]
First, suppose that $G$ satisfies one of (1-1),(1-2),(2-1)--(2-6). 
Then we can see that $\Cl(\kk[G])$ is isomorphic to $\Cl(\kk[K_{s_1,s_2}])$, $\Cl(\kk[K_{s_1,s_2}]) \oplus \Cl(\kk[K_{t_1,t_2}])$, 
$\Cl(\kk[K_{s_1,s_2}^{t_1,t_2}])$ or $\Cl(\kk[K_{1,s_1,s_2}^{t_1,t_2}])$, 
and those are isomorphic to $\ZZ$ or $\ZZ^2$ by Theorem~\ref{thm:bip_class_group}. 

\medskip

\noindent
(1) Since $v\in V(G)\setminus V(C_m)$ is regular, that is, $|\Psi_r|\ge d-(2k_m+1)$ and $|\Psi_f|\ge 2k_m+1$ by Lemma~\ref{fundamental}, 
we see that $G$ should contain one extra fundamental set or one extra regular vertex. 

Suppose that $G$ contains one extra fundamental. 
Then $p_{m,0},\ldots, p_{m,2k_m}$ are non-regular and we have $C_1=\cdots =C_m$ by Lemma~\ref{lem:regular} (1). 
By $G\ne C_m$, there exists a primitive odd path $P=x_0x_1\cdots x_l$ whose end vertices $x_0,x_l$ are in $V(C_m)$ and $x_k\notin V(C_m)$ for all $k\in [l-1]$. 
Furthermore, from Lemma~\ref{lem:regular} (2) and (3), we can see that vertices on $C_m$ whose degree are at least $3$ are just only $x_0$ and $x_l$. 
We may assume that $\{x_0,x_l\}=\{p_0,p_{2k_m}\}$. Consider the path $Q=p_{m,0} p_{m,1} \cdots p_{m,2k_m}$ and the graph $G'$ given by removing $Q^{\circ}$ from $G$. 
We can see that $G'$ contains no odd cycles, that is, $G'$ is bipartite and the edges on $Q$ does not appear as generators of toric ideal of $\kk[G]$. 
Since $\Cl(\kk[G]) \cong \Cl(\kk[G']) \cong \ZZ$, $G'$ is a complete bipartite graph $K_{s_1,s_2}$ with $s_1,s_2\ge 2$ 
by Theorem~\ref{thm:bip_class_group} and we see that $G$ is obtained by (1-1).

Suppose that $G$ has one extra regular vertex. We may assume that it is $p_{m,0}$. 
As above, by Lemma~\ref{lem:regular}, we can observe that $\{p_{m,1},p_{m,2},\ldots,p_{m,2k_m}\} \subset V(C_i)$ for all $i\in [m]$ and 
so vertices on $C_m$ whose degree are at least $3$ are just only $p_{m,2k_m}$, $p_{m,0}$ and $p_{m,1}$. 
Consider the path $Q=p_{m,1} p_{m,2} \cdots p_{m,2k_m}$ and the graph $G'$ given by removing $Q^{\circ}$ from $G$. 
We can see that $G'$ has no odd cycles, that is, $G'$ is bipartite and the edges on $Q$ does not appear as generators of toric ideal of $\kk[G]$. 
Since $\Cl(\kk[G]) \cong \Cl(\kk[G']) \ZZ$, $G'$ is a complete bipartite graph $K_{s_1,s_2}$ with $s_1,s_2\ge 2$ by Theorem~\ref{thm:bip_class_group} 
and we see that $G$ is obtained by (1-2). 

\medskip

\noindent
(2) Similarly to (1), $G$ has 
\begin{itemize}
\item[(i)] two extra fundamental sets, 
\item[(ii)] one extra vertex and one extra fundamental set, or 
\item[(iii)] two extra regular vertices.
\end{itemize}

Suppose that (i). 
Then $p_{m,0},\ldots, p_{m,2k_m}$ are non-regular and we have $C_1=\cdots =C_m$ by Lemma~\ref{lem:regular} (1). 
If there exists just one type of paths $P_i=x_{i,0}\cdots x_{i,l_i}$ 
whose end vertices $x_{i,0},x_{l_i}$ are in $V(C_m)$ and $x_{i,k}\notin V(C_m)$ for all $k\in [l_i-1]$, $G$ is obtained by (2-4). 
Suppose that there exist two types of paths $P_1,P_2$. 
We may assume that $\{x_{1,0},x_{1,l_1}\}=\{p_{m,0},p_{m,1}\}$ and $\{x_{2,0},x_{2,l_2}\}=\{p_{m,j},p_{m,j+1}\}$. 
Consider two paths $Q_1=p_{m,0} \cdots p_{m,j}$ and $Q_2=p_{m,j+1} \cdots p_{m,2k_m} p_{m,0}$ and 
the graph $G'$ given by removing $Q_1^{\circ}$ and $Q_2^{\circ}$ from $G$. 
We can observe that $G'$ has two connected components $G_1,G_2$ and they have no odd cycles, that is, they are bipartite. 
Therefore, we have $\Cl(\kk[G]) \cong \Cl(\kk[G_1])\oplus \Cl(\kk[G_2]) \cong \ZZ^2$ and so $G_1,G_2$ are complete bipartite graphs $K_{s_1,s_2},K_{t_1,t_2}$ with $s_1,s_2,t_1,t_2\ge 2$. 
This $G$ is obtained by (2-1). 

Suppose that (ii). We may assume that it is $p_{m,0}$. 
We observe that $\{p_{m,1},\ldots,p_{m,2k_m}\} \subset V(C_i)$ for all $i\in [m]$, and $p_{m,2k_m}$, $p_{m,0}$ and $p_{m,1}$ have degree $3$ or more. 
If the other vertices have degree $2$, then $G$ is obtained by (2-5). 
If there exist the other vertices whose degree is at least $3$, then there exists a primitive odd path $P=x_0 \cdots x_l$ 
with end vertices $\{x_0,x_l\}=\{p_{m,j},p_{m,j+1}\}$ for $j\in [2k_m-1]$. Then this $G$ is obtained by (2-2). 

Suppose that (iii). We may assume that $p_{m,0}$ and $p_{m,j}$ are regular. 
If $k_1<k_m$, $k_1=k_m-1$ because $\{p_{m,1},\ldots,\hat{p}_{m,j},\ldots,p_{m,2k_m}\} \subset C_i$ for all $i\in [m]$. 
However, then $C_m$ has a chord, a contradiction. Thus, $k_1=k_m$. 
If $j\ne 1,2k_m$, the vertices on $C_m$ whose degree are at least $3$ are $p_{m,2k_m},p_{m,0},p_{m,1},p_{m,j-1},p_{m,j}$ and $p_{m,j+1}$. 
This $G$ is obtained by (2-3). 

Suppose that $j=1$ or $2k_m$. 
We may assume that $j=1$. If $k_m\ge 2$, the vertices on $C_m$ whose degree are at least $3$ are $p_{m,2k_m},p_{m,0},p_{m,1},p_{m,2}$. 
Hence, This $G$ is obtained by (2-4). 

Suppose that $j=1$ and $k_m=1$. Note that $G\setminus p_{m,2}$ is bipartite. Let $V_1$ and $V_2$ be the partition of the bipartite graph $G\setminus p_{m,2}$, 
let $S_i=N_G(p_{m,2})\cap V_i$ for $i=1,2$ and let $T_i=V_i\setminus U_i$. 
We show that $G\setminus p_{m,2}$ coincides with $K_{s_1,s_2}^{t_1,t_2}$, where $s_i=|S_i|\ge 2$ and $t_i=|T_i|$ for $i=1,2$. 

%Note that all vertices except for $p_{m,2}$, and $V_1$ and $V_2$ are fundamental sets since $G\setminus p_{m,2}$ is connected. 
%If $\{v_1,v_2\}\notin E(G)$ for some $v_1\in S_1$, $v_2\in S_2$, then $v_1$ and $v_2$ are connected by a primitive path of at least $2$ lengths on $G\setminus p_{m,2}$, 
%that is, $G$ has the primitive odd cycle of length at least $5$, a contradiction. 
%If $\{u_1,u_2\}\in E(G)$ for some $u_1\in T_1$, $u_2\in T_2$, $\{p_{m,2},u_i\}$ is an independent set and 
%we can obtain an independent set $I_i$ by adding $\{p_{m,2},u_i\}$ to some vertices in $T_i$ such that $B(T)$ is connected for $i=1,2$. 
%It is a contradiction to $\Cl(\kk[G]) \cong \ZZ^2$. Then we can see that $\{p_{m,2}\}\cup T_1\cup T_2$ is fundamental. 
%Finally, if $\{w_1,w_2\}\notin E(G)$ for some $w_1\in T_1$ and $w_2\in S_2$, then $\{w_1,w_2\}$ is an independent set and 
%we can obtain an independent set $I$ by adding $\{w_1,w_2\}$ to some vertices in $S_2$ such that $B(T)$ is connected, a contradiction by the same reason. 
%Therefore, $G$ satisfies (2-6). 
Note that all vertices except for $p_{m,2}$ are regular, $V_1$ and $V_2$ are fundamental sets since $G\setminus p_{m,2}$ is connected, and there exists a fundamental set $T$ containing $p_{m,2}$. 
If $\{v_1,v_2\}\notin E(G)$ for some $v_1\in S_1$, $v_2\in S_2$, then $\{v_1 v_2\}$ is an independent set and $B(\{v_1,v_2\})$ is connected. Thus, we can obtain a fundamental set containing $\{v_1 v_2\}$ and it is different from $V_1,V_2,T$. It is a contradiction to $\Cl(\kk[G]) \cong \ZZ^2$.
If $\{u_1,u_2\}\in E(G)$ for some $u_1\in T_1$, $u_2\in T_2$, $\{p_{m,2},u_i\}$ is an independent set and 
we can obtain an independent set $I_i$ by adding $\{p_{m,2},u_i\}$ to some vertices in $T_i$ such that $B(I_i)$ is connected for $i=1,2$, a contradiction by the same reason.  Then we have $T=\{p_{m,2}\}\cup T_1\cup T_2$. 
Finally, if $\{w_1,w_2\}\notin E(G)$ for some $w_1\in T_1$ and $w_2\in S_2$, then $\{w_1,w_2\}$ is an independent set and 
we can obtain an independent set $I$ by adding $\{w_1,w_2\}$ to some vertices in $S_2$ such that $B(I)$ is connected, a contradiction by the same reason. 
Therefore, $G$ satisfies (2-6). 
\end{proof}

\bigskip
%%%%%%%%%%%%%%%%%%%%%%%%%%%%%%%%%%%%%%%%%%%%%%%%%%%%%%%%%%%%%%%%%%%%%%%%%%%%%%%%%%%%%%%%%%%%%%%%%%%%%%%%%%%%%%%%%%%%%%%%%%%%%%%
%%%%%%%%%%%%%%%%%%%%%%%%%%%%%%%%%%%%%%%%%%%%%%%%%%%%%%%%%%%%%%%%%%%%%%%%%%%%%%%%%%%%%%%%%%%%%%%%%%%%%%%%%%%%%%%%%%%%%%%%%%%%%%%
%%%%%%%%%%%%%%%%%%%%%%%%%%%%%%%%%%%%%%%%%%%%%%%%%%%%%%%%%%%%%%%%%%%%%%%%%%%%%%%%%%%%%%%%%%%%%%%%%%%%%%%%%%%%%%%%%%%%%%%%%%%%%%%
%%%%%%%%%%%%%%%%%%%%%%%%%%%%%%%%%%%%%%%%%%%%%%%%%%%%%%%%%%%%%%%%%%%%%%%%%%%%%%%%%%%%%%%%%%%%%%%%%%%%%%%%%%%%%%%%%%%%%%%%%%%%%%%
%%%%%%%%%%%%%%%%%%%%%%%%%%%%%%%%%%%%%%%%%%%%%%%%%%%%%%%%%%%%%%%%%%%%%%%%%%%%%%%%%%%%%%%%%%%%%%%%%%%%%%%%%%%%%%%%%%%%%%%%%%%%%%%
%%%%%%%%%%%%%%%%%%%%%%%%%%%%%%%%%%%%%%%%%%%%%%%%%%%%%%%%%%%%%%%%%%%%%%%%%%%%%%%%%%%%%%%%%%%%%%%%%%%%%%%%%%%%%%%%%%%%%%%%%%%%%%%

\section{The relationships among ${\bf Order}_n$, ${\bf Stab}_n$ and ${\bf Edge}_n$}\label{sec:relation}

Recall that ${\bf Order}_n$, ${\bf Stab}_n$ and ${\bf Edge}_n$ are the sets of unimodular equivalence classes 
of order polytopes, stable set polytopes and edge polytopes such that the associated toric rings have the class groups of rank $n$, respectively. 
This section is devoted to the discussions on the relationships among ${\bf Order}_n$, ${\bf Stab}_n$ and ${\bf Edge}_n$ 
in the cases $n=1,2,3$ by using the results in the previous section. 

\medskip

\subsection{The case $n=1$}\label{sec:n=1}

\begin{prop}
%Let $R=\kk[\Pi]$ or $\kk[\Stab_G]$ or $\kk[G]$ for some poset $\Pi$ or some graph $G$. 
Let $R$ be the Segre product of the polynomial rings $\kk[x_1,\ldots,x_s]$ and $\kk[y_1,\ldots,y_t]$ for some $s,t \in \ZZ_{>0}$. 
Note that $\Cl(R) \cong \ZZ$. Then $R$ is isomorphic to $\kk[\Pi]$, $\kk[\Stab_G]$ and $\kk[H]$ for some poset $\Pi$ and some graphs $G,H$. 

Conversely, for $S=\kk[\Pi]$ or $\kk[\Stab_G]$ or $\kk[H]$ for some poset $\Pi$ or some graphs $G,H$ with $\Cl(S) \cong \ZZ$ such that $S$ is not a polynomial extension, 
$S$ is isomorhic to the Segre product of the polynomial rings $\kk[x_1,\ldots,x_s]$ and $\kk[y_1,\ldots,y_t]$ for some $s,t \in \ZZ_{>0}$. 

In particular, we have ${\bf Order}_1={\bf Stab}_1={\bf Edge}_1$. 
\end{prop}
\begin{proof}
These statements follow from Proposition~\ref{char:order} (1), Theorems~\ref{thm:stab} (1), \ref{thm:bip_class_group} (1) and \ref{thm:nonbip_class_group} (1). 
Note that the edge polytope $P_{K_{s_1+1,s_2+1}}$ is unimodularly equivalent to the order polytope $\calO_{\Pi_1(s_1,s_2)}$ (see \cite{HM}). 
Moreover, the procedures (1-1) and (1-2) in Theorem~\ref{thm:nonbip_class_group} (1) correspond to the lattice pyramid construction. 
\end{proof}

\medskip

\subsection{The case $n=2$}\label{sec:n=2}

\begin{lem}\label{lem:order} 
Let $s_1,s_2,t_1,t_2$ be positive integers and let $d=s_1+s_2+t_1+t_2$.
\begin{itemize} 
\item[(1)] The edge polytope $P_{K_{s_1+1,s_2+1}^{t_1,t_2}}$ is unimodularly equivalent to the order polytope $\calO_{\Pi_3(s_1,s_2,t_1,t_2,0)}$. 
\item[(2)] The edge polytope $P_{K_{1,s_1+1,s_2+1}^{t_1-1,t_2-1}}$ is unimodularly equivalent to the order polytope $\calO_{\Pi_3(s_1,s_2,t_1,t_2,0)}$. 
\end{itemize}
In particular, $P_{K_{s_1+1,s_2+1}^{t_1,t_2}}$ and $P_{K_{1,s_1+1,s_2+1}^{t_1-1,t_2-1}}$ are unimodularly equivalent. 
\end{lem}
\begin{proof}
It is enough to show that $P_{K_{s_1+1,s_2+1}^{t_1,t_2}}$ (resp. $P_{K_{1,s_1+1,s_2+1}^{t_1-1,t_2-1}}$) is unimodularly equivalent to 
$\calC(\Pi_3(s_1,s_2,t_1,t_2,0))$  (resp. $\calC(\Pi_3(s_1,s_2,t_1,t_2,0))$). 

\medskip

\noindent
(1) By Definition~\ref{bipartite}, it is straightforward to see that the vertices of $P_{K_{s_1+1,s_2+1}^{t_1,t_2}}$ one-to-one correspond to the antichains of $\Pi_3(s_1,s_2,t_1,t_2,0)$ 
by considering the projection $\RR^{d+2} \rightarrow \RR^d$ which ignores the $1$-th and $d$-th coordinates 
and this projection gives a unimodular transformation between $P_{K_{s_1+1,s_2+1}^{t_1,t_2}}$ and $\calC(\Pi_3(s_1,s_2,t_1,t_2,0))$.

\medskip

\noindent
(2) Consider the projection $\RR^{d+1} \rightarrow \RR^d$ by ignoring the $(d+1)$-th coordinate. 
Then the set of vertices of $P_{K_{1,s_1+1,s_2+1}^{t_1-1,t_2-1}}$ becomes $\{\eb_i+\eb_j : 1 \leq i \leq s_1+t_1, s_1+t_1+t_2 \leq j \leq d\} \cup \{\eb_i+\eb_j : 1 \leq i \leq s_1+1, s_1+t_1+1 \leq j \leq d\} \cup \{\eb_k : 1 \leq k \leq s_1+1 \text{ or } s_1+t_1+t_2 \leq k \leq d\}$. 
By applying a unimodular transformation $\begin{pmatrix}
1 &1 &\cdots &1 &  &  &       & \\
  &1 &       &  &  &  &       & \\
  &  &\ddots &  &  &  &       & \\
  &  &       &1 &  &  &       & \\
  &  &       &  &1 &  &       & \\
  &  &       &  &  &1 &       & \\
  &  &       &  &  &  &\ddots & \\
  &  &       &  &1 &1 &\cdots &1 
\end{pmatrix}$ to those vertices (from the left-hand side) and translating them by $-\eb_1-\eb_d$ and applying a unimodular transformation $\begin{pmatrix}
-1 &  &       &  & \\
   &1 &       &  & \\
   &  &\ddots &  & \\
   &  &       &1 & \\
   &  &       &  &-1
\end{pmatrix}$, the vertices become as follows: \begin{align*}
&\eb_i+\eb_j \mapsto \eb_1+\eb_i+\eb_j+\eb_d \mapsto \eb_i+\eb_j \mapsto \eb_i+\eb_j \\
&(1< i \leq s_1+t_1, \; s_1+t_1+t_2 \leq j < d \text{ or } 1< i \leq s_1+1, \; s_1+t_1+1 \leq j < d) \\
&\eb_i+\eb_d \mapsto \eb_1+\eb_i+\eb_d \mapsto \eb_i \mapsto \eb_i \; (1 < i \leq s_1+t_1) \\
&\eb_1+\eb_j \mapsto \eb_1+\eb_j+\eb_d \mapsto \eb_j \mapsto \eb_j \; (s_1+t_1+1 \leq j < d), \quad
\eb_1+\eb_d \mapsto {\bf 0} \\
&\eb_k \mapsto \eb_1+\eb_k \mapsto \eb_k-\eb_d \mapsto \eb_k+\eb_d \; (1 < k \leq s_1+1), \;\;
\eb_k \mapsto \eb_1+\eb_k \;(s_1+t_1+t_2 \leq k < d) \\
&\eb_1 \mapsto \eb_d, \quad \eb_d \mapsto \eb_1.
\end{align*}
We can directly see that these lattice points one-to-one correspond to the antichains of $\Pi_3(s_1,s_2,t_1,t_2,0)$. 
\end{proof}

\begin{prop}
{\em (1)} Let $G$ be a perfect graph with $\Cl(\kk[\Stab_G]) \cong \ZZ^2$. 
Then $\Stab_G$ is unimodularly equivalent to $\calO_\Pi$ for some poset $\Pi$. 
In particular, we have ${\bf Stab}_2 \subset {\bf Order}_2$. \\
{\em (2)} Let $G$ be a $2$-connected graph with $\Cl(\kk[G]) \cong \ZZ^2$. 
Then $P_G$ is unimodularly equivalent to $\calO_\Pi$ for some poset $\Pi$. 
In particular, we have ${\bf Edge}_2 \subset {\bf Order}_2$. \\
{\em (3)} Let $\Pi$ be a poset with $\Cl(\kk[\Pi]) \cong \ZZ^2$. 
Then $\calO_\Pi$ is unimodularly equivalent to $\calC_{G(\Pi)}$ or $P_G$ for some $G$. 
In particular, ${\bf Order}_2 \subset {\bf Stab}_2 \cup {\bf Edge}_2$. \\
{\em (4)} There exist a graph $G$ and a graph $H$ with $\Cl(\kk[\Stab_G]) \cong \Cl(\kk[H]) \cong \ZZ^2$ 
such that $\Stab_G \not\in {\bf Edge}_2$ and $P_H \not\in {\bf Stab}_2$, respectively. 
\end{prop}
\begin{proof}
The statement (1) directly follows from Theorem~\ref{thm:stab} (2). 
The statement (2) follows from Theorems~\ref{thm:bip_class_group} (2), \ref{thm:nonbip_class_group} (2) and Lemma~\ref{lem:order}. 

\medskip

\noindent
(3) By Propositions~\ref{char:order} and \ref{prop:compare}, it is enough to consider the case $\Pi=\Pi_4(s_1,s_2,t_1,t_2)$ for some $s_1,s_2,t_1,t_2\in \ZZ_{>0}$. 
Let $K$ be the bipartite graph on the vertex set $[d+3]$ with the edge set 
\begin{align*}
E(K)&=\{\{i,j\} : i \in \{1,\ldots, t_1, d+2\}, j\in \{t_1+1,\ldots, t_1+t_2, d+3\} \text{ or }\\ 
    &i\in \{t_1+t_2+1,\ldots, t_1+t_2+s_1, d+3\}, j\in \{t_1+t_2+s_1+1,\ldots, d, d+1\}\}.
\end{align*}
Note that $K$ is obtained by identifying some vertex of $K_{s_1+1,s_2+1}$ and some vertex of $K_{t_1+1,t_2+1}$ (see Figure~\ref{graphK}). 

Moreover, let $I_p=\{q\in \Pi_4 : q \prec p\}$ for $p\in \Pi_4$. Note that for any poset ideal $I$ of $\Pi_4$, 
$I$ coincides with the empty set, $I_p$ or $I_p\cup I_q$ for some $p,q\in \Pi_4$. 
We can see that by consider the projection $\RR^{d+3} \to \RR^{d+1}$ ignoring the ($d+2$)-th and ($d+3$)-th coordinates and 
by applying a unimodular transformation 
\scalebox{0.6}{
$\begin{pmatrix}
1 &\cdots  &1        & & & & & & &                                           1& \cdots & 1& 1\\
  & \ddots & \vdots  & & & & & & &                                           1&\cdots & 1 &1 \\
  &        &1        & &       &  & & & &                                     &       &   & \\
  &        &         &1&\cdots &1      & & & &                              \vdots &\vdots & \vdots & \vdots \\
  &        &         & &\ddots &\vdots & & & &                                   & &  & \\
  &        &         & &       &1      & &       &       &           1 &\cdots &1 &1 \\
  &        &         & &       &       &1&\cdots &1      &             &       &  & \\
  &        &         & & &  &   &\ddots &\vdots &   &       &  & \\
  &        & & & & & & &1& &  &  &  \\
  &        & & & & & & &  &1      &\cdots &1 & \\
  &        & & & & & & & &  &\ddots & \vdots&  \\
  &        & & & & & & & &  &       & 1     &  \\
  &        & & & & & & & & 1& \cdots& 1     &1 
\end{pmatrix}$}
to vertices of $P_K$, the vertices become as follows: \begin{align*}
&\eb_i+\eb_{d+3} \mapsto \eb_i \mapsto \sum_{p_k\in I_{p_i}}\eb_k  \; (1\leq i \leq t_1\text{ or } t_1+t_2+s_1+1\leq i \leq d+1), \\
&\eb_i+\eb_{d+2} \mapsto \eb_i \mapsto \sum_{p_k\in I_{p_i}}\eb_k  \; (t_1+1\leq i \leq t_1+t_2), \quad \eb_{d+2}+\eb_{d+3} \mapsto 0, \\
&\eb_i+\eb_{d+1} \mapsto \sum_{p_k\in I_{p_i}}\eb_k  \; (t_1+t_2+1\leq i \leq t_1+t_2+s_1), \\
&\eb_i+\eb_j \mapsto \sum_{p_k\in I_{p_i}\cup I_{p_j}}\eb_k, \\
&(1\leq i \leq t_1, \; t_1+1\leq j \leq t_1+t_2\text{ or } t_1+t_2+1\leq i \leq t_1+t_2+s_1, \; t_1+t_2+s_1+1\leq j \leq d). 
\end{align*}
We can directly see that these lattice points one-to-one correspond to the poset ideals of $\Pi_4(s_1,s_2,t_1,t_2)$.

\medskip

\noindent
(4) Let $G=G(\Pi_2(1,1,1,2))$ (see Figure~\ref{graph1112}) and let $H$ be the graph on the vertex set $\{1,\ldots,7\}$ 
with the edge set $E(G)=\{12,17,26,34,47,56,57,67\}$ (see Figure~\ref{edgeH}). 
Then we have $\Cl(\kk[\Stab_G]) \cong \Cl(\kk[P_G]) \cong \ZZ^2$ by construction. %$\Stab_G \notin {\bf Edge}_2$, $P_H\notin {\bf Stab}_2$.

If $\Stab_G \in {\bf Edge}_2$, that is, there exists a graph $G'$ such that $P_{G'}$ is unimodularly equivalent to $\Stab_G$, 
then $G'$ satisfies that $G'$ is bipartite and has $7$ vertices and $12$ edges or $G'$ is non-bipartite and has $6$ vertices and $12$ edges. 
We can check by {\tt MAGMA} that for any such graphs $G'$, $P_{G'}$ is not unimodularly equivalent to $\Stab_G$. 

Similarly, if $P_H \in {\bf Stab}_2$, that is, there exists a graph $H'$ such that $\Stab_{H'}$ is unimodularly equivalent to $P_G$, 
then $H'$ has 5 vertices and 8 independent sets. Similarly, we can check by {\tt MAGMA} that 
for any such graphs $H'$, $\Stab_{H'}$ is not unimodularly equivalent to $P_H$. 
\end{proof}

%%%%%%%%%%%%%%%%%%%%%
\begin{figure}[h]%n=2
{\scalebox{1.0}{
\begin{minipage}{0.8\columnwidth}
\centering
{\scalebox{0.6}{
\begin{tikzpicture}[line width=0.02cm]%(b3)

\coordinate (N11) at (0,0); 
\coordinate (N12) at (0,1); 
\coordinate (N13) at (0,2); 
\coordinate (N14) at (0,3);
\coordinate (N15) at (0,4); 

\coordinate (N21) at (2,0); 
\coordinate (N22) at (2,1); 
\coordinate (N23) at (2,2); 
\coordinate (N24) at (2.1,2.9);
\coordinate (N25) at (2.5,4); 

\coordinate (N31) at (3,0); 
\coordinate (N32) at (3,1); 
\coordinate (N33) at (3,2); 
\coordinate (N34) at (2.9,2.9);
\coordinate (N35) at (2.5,4); 

\coordinate (N41) at (5,0); 
\coordinate (N42) at (5,1); 
\coordinate (N43) at (5,2); 
\coordinate (N44) at (5,3);
\coordinate (N45) at (5,4);

%edge
 
\draw[dotted]  (0,1.8)--(0,2.2);
\draw[dotted]  (2,1.8)--(2,2.2);
\draw[dotted]  (3,1.8)--(3,2.2);
\draw[dotted]  (5,1.8)--(5,2.2);

\draw  (N15)--(N21);
\draw  (N15)--(N22);
\draw  (N15)--(N24);
\draw  (N15)--(N25);

\draw  (N14)--(N21);
\draw  (N14)--(N22);
\draw  (N14)--(N24);
\draw  (N14)--(N25);

\draw  (N12)--(N21);
\draw  (N12)--(N22);
\draw  (N12)--(N24);
\draw  (N12)--(N25);

\draw  (N11)--(N21);
\draw  (N11)--(N22);
\draw  (N11)--(N24);
\draw  (N11)--(N25);

\draw  (N35)--(N41);
\draw  (N35)--(N42);
\draw  (N35)--(N44);
\draw  (N35)--(N45);

\draw  (N34)--(N41);
\draw  (N34)--(N42);
\draw  (N34)--(N44);
\draw  (N34)--(N45);

\draw  (N32)--(N41);
\draw  (N32)--(N42);
\draw  (N32)--(N44);
\draw  (N32)--(N45);

\draw  (N31)--(N41);
\draw  (N31)--(N42);
\draw  (N31)--(N44);
\draw  (N31)--(N45);

%node

\draw [line width=0.05cm, fill=white] (N11) circle [radius=0.15];% node[below left] {\Large $j$}; 
\draw [line width=0.05cm, fill=white] (N12) circle [radius=0.15];% node[] at (-1.5,1) {\large $s_1+t_1-1$};
\draw [line width=0.05cm, fill=white] (N14) circle [radius=0.15];% node[] at (-1,3) {\large $s_1+1$};
\draw [line width=0.05cm, fill=white] (N15) circle [radius=0.15] node[] at (0,4.5) {\Large $d+2$};

\draw [line width=0.05cm, fill=white] (N21) circle [radius=0.15]; %node[below right] {\Large $j$};
\draw [line width=0.05cm, fill=white] (N22) circle [radius=0.15];% node[] at (5.5,1) {\large $s_1+t_1+2$};
\draw [line width=0.05cm, fill=white] (N24) circle [radius=0.15];% node[] at (5.5,3) {\large $s_1+t_1+t_2$};
\draw [line width=0.05cm, fill=white] (N25) circle [radius=0.15] node[] at (2.5,4.5) {\Large $d+3$};

\draw [line width=0.05cm, fill=white] (N31) circle [radius=0.15]; %node[below left] {\Large $l$}; 
\draw [line width=0.05cm, fill=white] (N32) circle [radius=0.15];% node[] at (-1.5,1) {\large $s_1+t_1-1$};
\draw [line width=0.05cm, fill=white] (N34) circle [radius=0.15];% node[] at (-1,3) {\large $s_1+1$};
\draw [line width=0.05cm, fill=white] (N35) circle [radius=0.15]; %node[below left] {\large $k$};

\draw [line width=0.05cm, fill=white] (N41) circle [radius=0.15];% node[below right] {\Large $l$};
\draw [line width=0.05cm, fill=white] (N42) circle [radius=0.15];% node[] at (5.5,1) {\large $s_1+t_1+2$};
\draw [line width=0.05cm, fill=white] (N44) circle [radius=0.15];% node[] at (5.5,3) {\large $s_1+t_1+t_2$};
\draw [line width=0.05cm, fill=white] (N45) circle [radius=0.15] node[] at (5,4.5) {\Large $d+1$};  
\node at (5,0) {$$};

\end{tikzpicture}
}}
\caption{The graph $K$}
\label{graphK}
\end{minipage}
}}
\begin{minipage}{0.49\columnwidth}
\centering
{\scalebox{0.7}{
\begin{tikzpicture}[line width=0.05cm]

\coordinate (N14) at (1,0); 
\coordinate (N15) at (1,1); 
\coordinate (N16) at (-0.5,2); 
\coordinate (N17) at (2.5,2);

\coordinate (Nt1) at (3.5,1);

%edge

\draw (N14)--(N15);
\draw (N15)--(N16);
\draw (N15)--(N17);
\draw (N14)--(N16);
\draw (N14)--(N17);

%node

\draw [line width=0.05cm, fill=white] (N14) circle [radius=0.15] node[] at (1,-0.5) {\Large $1$};
\draw [line width=0.05cm, fill=white] (N16) circle [radius=0.15] node[above left] {\Large $3$};
\draw [line width=0.05cm, fill=white] (N15) circle [radius=0.15] node[] at (1,1.5) {\Large $2$};
\draw [line width=0.05cm, fill=white] (N17) circle [radius=0.15] node[above right]  {\Large $4$};

\draw [line width=0.05cm, fill=white] (Nt1) circle [radius=0.15] node[below right] {\Large $5$};

%\node at (8,0) {$$}; 

\end{tikzpicture} }}
\caption{\\ \hspace{-0.4cm} The graph $G(\Pi_2(1,1,1,2))$}
\label{graph1112}
\end{minipage}
\begin{minipage}{0.49\columnwidth}
\centering
{\scalebox{0.7}{
\begin{tikzpicture}[line width=0.05cm]

\coordinate (4) at (2.5,2); 

\coordinate (1) at (0,0); 
\coordinate (3) at (0,2); 

\coordinate (2) at (2,0);

\coordinate (5) at (5,0); 
\coordinate (7) at (5,2); 

\coordinate (6) at (3,0);

%edge

\draw (1)--(2);
\draw (1)--(4);
\draw (3)--(2);
\draw (3)--(4);

\draw (4)--(5);
\draw (4)--(7);
\draw (6)--(5);
\draw (6)--(7);

%node

\draw [line width=0.05cm, fill=white] (1) circle [radius=0.15]  node[below left] {\Large $1$};
\draw [line width=0.05cm, fill=white] (2) circle [radius=0.15]  node[below left] {\Large $2$};
\draw [line width=0.05cm, fill=white] (3) circle [radius=0.15]  node[below left] {\Large $6$}; 
\draw [line width=0.05cm, fill=white] (4) circle [radius=0.15]  node[] at (2.5, 1.5) {\Large $7$}; 
\draw [line width=0.05cm, fill=white] (5) circle [radius=0.15]  node[below right] {\Large $4$}; 
\draw [line width=0.05cm, fill=white] (6) circle [radius=0.15]  node[below right] {\Large $3$}; 
\draw [line width=0.05cm, fill=white] (7) circle [radius=0.15]  node[below right] {\Large $5$};

%\node at (8,0) {$$}; 

\end{tikzpicture} }}
\caption{The graph $H$}
\label{edgeH}
\end{minipage}
\end{figure}
%%%%%%%%%%%%%%%%%%%%%

\bigskip

\subsection{The case $n=3$}\label{sec:n=3}
We conclude the present paper by providing examples showing that 
there is no inclusion among ${\bf Order}_3$, ${\bf Stab}_3$ and ${\bf Edge}_3$. 

We define the following three objects: a poset $\Pi$, a perfect graph $\Gamma$ and a connected graph $G$. 
\begin{itemize}
\item Let $\Pi=\{z_1,\ldots,z_6\}$ equipped with the partial orders $z_1 \prec z_3 \prec z_4$ and $z_2 \prec z_3 \prec z_5$. 
Namely, $\Pi$ is the disjoint union of the ``X-shape'' poset and one point. See Figure~\ref{posetX}. 
Then we see from \eqref{eq:Hibi} that $\Cl(\kk[\Pi]) \cong \ZZ^3$. 
\item Let $\Gamma$ be the graph on the vertex set $\{1,\ldots,6\}$ with the edge set 
$$E(\Gamma)=\{15,16,24,26,34,35,45,46,56\},$$
See Figure~\ref{graphG}. 
Then $\Gamma$ is perfect since $\Gamma$ is chordal. Moreover, $\Gamma$ contains four maximal cliques: $\{1,5,6\},\{2,4,6\},\{3,4,5\}$ and $\{4,5,6\}$. 
Thus, we see that $\Cl(\kk[\Stab_\Gamma]) \cong \ZZ^3$. 
\item Let $G=K_{2,2,2}$ be the complete tripartite graph. Namely, $V(G)=\{1,\ldots,6\}$ with $$E(G)=\{13,14,15,16,23,24,25,26,35,36,45,46\}.$$
See Figure~\ref{K222}. The class groups of the edge rings of complete multipartite graphs are investigated in \cite{HM}. 
By \cite[Theorem 1.3]{HM}, we see that $\Cl(\kk[G]) \cong \ZZ^3$. 
\end{itemize}

%%%%%%%%%%%%%%%%%%%%%
\begin{figure}[h]%n=3

\begin{minipage}{0.32\columnwidth}
\centering
{\scalebox{0.7}{
\begin{tikzpicture}[line width=0.05cm]%(b1)

\coordinate (1) at (0,0); 
\coordinate (2) at (2,0); 
\coordinate (3) at (1,1); 
\coordinate (4) at (0,2);
\coordinate (5) at (2,2); 

\coordinate (6) at (3,1);

%edge
 
\draw (1)--(3);
\draw (2)--(3);
\draw (3)--(4);
\draw (3)--(5);

%node

\draw [line width=0.05cm, fill=white] (1) circle [radius=0.15] node[below left] {\Large $1$}; 
\draw [line width=0.05cm, fill=white] (2) circle [radius=0.15] node[below right] {\Large $2$};
\draw [line width=0.05cm, fill=white] (3) circle [radius=0.15] node[] at (1,0.5) {\Large $3$};
\draw [line width=0.05cm, fill=white] (4) circle [radius=0.15] node[below left] {\Large $4$};
\draw [line width=0.05cm, fill=white] (5) circle [radius=0.15] node[below right] {\Large $5$};
\draw [line width=0.05cm, fill=white] (6) circle [radius=0.15] node[] at (3,0.5) {\Large $6$};

%\node at (5,0) {$$}; 

\end{tikzpicture}
} }
\caption{The poset $\Pi$}
\label{posetX}
\end{minipage}
\begin{minipage}{0.32\columnwidth}
\centering
{\scalebox{0.7}{
\begin{tikzpicture}[line width=0.05cm]%(b2)

\coordinate (1) at (0,0); 
\coordinate (6) at (-1.73,1); 
\coordinate (2) at (-1.73,3); 
\coordinate (4) at (0,4);
\coordinate (3) at (1.73,3); 
\coordinate (5) at (1.73,1); 

%edge

\draw (1)--(5);
\draw (1)--(6);
\draw (2)--(4);
\draw (2)--(6);
\draw (3)--(4);
\draw (3)--(5);
\draw (4)--(5);
\draw (4)--(6);
\draw (5)--(6);

%node

\draw [line width=0.05cm, fill=white] (1) circle [radius=0.15] node[] at (0,-0.5) {\Large 1}; 
\draw [line width=0.05cm, fill=white] (2) circle [radius=0.15] node[above left] {\Large 2};
\draw [line width=0.05cm, fill=white] (3) circle [radius=0.15] node[above right] {\Large 3};
\draw [line width=0.05cm, fill=white] (4) circle [radius=0.15] node[] at (0,4.5) {\Large 4};
\draw [line width=0.05cm, fill=white] (5) circle [radius=0.15] node[below right] {\Large 5};
\draw [line width=0.05cm, fill=white] (6) circle [radius=0.15] node[below left] {\Large 6};

%\node at (5,0) {$$}; 

\end{tikzpicture}
} }
\caption{The graph $\Gamma$}
\label{graphG}
\end{minipage}
\begin{minipage}{0.32\columnwidth}
\centering
{\scalebox{0.7}{
\begin{tikzpicture}[line width=0.05cm]%(b3)

\coordinate (1) at (0.5,0); 
\coordinate (2) at (1,-1); 

\coordinate (3) at (4,-1); 
\coordinate (4) at (4.5,0); 

\coordinate (6) at (2,2); 
\coordinate (5) at (3,2); 

\draw (1)--(3);
\draw (1)--(4);
\draw (1)--(5);
\draw (1)--(6);
\draw (2)--(3);
\draw (2)--(4);
\draw (2)--(5);
\draw (2)--(6);
\draw (3)--(5);
\draw (3)--(6);
\draw (4)--(5);
\draw (4)--(6);

%node

\draw [line width=0.05cm, fill=white] (1) circle [radius=0.15] node[below left] {\Large 1}; 
\draw [line width=0.05cm, fill=white] (2) circle [radius=0.15] node[below left] {\Large 2};
\draw [line width=0.05cm, fill=white] (3) circle [radius=0.15] node[below right] {\Large 3};
\draw [line width=0.05cm, fill=white] (4) circle [radius=0.15] node[below right] {\Large 4};
\draw [line width=0.05cm, fill=white] (5) circle [radius=0.15] node[] at (3,2.5) {\Large 5};
\draw [line width=0.05cm, fill=white] (6) circle [radius=0.15] node[] at (2,2.5) {\Large 6};

%\node at (5,0) {$$}; 

\end{tikzpicture}
}}
\caption{The graph $K_{2,2,2}$}
\label{K222}
\end{minipage}
\end{figure}
%%%%%%%%%%%%%%%%%%%%%

We can see that $\calO_\Pi \not\in {\bf Stab}_3 \cup {\bf Edge}_3$, $\Stab_\Gamma \not\in {\bf Order}_3 \cup {\bf Edge}_3$ and 
$P_G \not\in {\bf Order}_3 \cup {\bf Stab}_3$ as follows. 

\noindent
\underline{$\calO_\Pi \not\in {\bf Stab}_3 \cup {\bf Edge}_3$}: Consider $\calO_\Pi$. 

Suppose that there exists a perfect graph $\Gamma'$ such that $\Stab_{\Gamma'}$ is unimodularly equivalent to $\calO_\Pi$. 
Then $\Gamma'$ has $6$ vertices and non-trivial $4$ independent sets. 
Since such graphs are finitely many, we can check by {\tt MAGMA} that their stable set polytopes are not unimodularly equivalent to $\calO_\Pi$. 

Similarly, suppose that there exists a graph $G'$ such that $P_{G'}$ is unimodularly equivalent to $\calO_\Pi$. 
Then $G'$ is a bipartite graph on $8$ vertices or a non-bipartite graph on $7$ vertices. 
Since $\Cl(\kk[G']) \cong \ZZ^3$, $G'$ contains at most one non-bipartite block by Proposition~\ref{rank}. 
We can also check that edge polytopes of such graphs are not unimodularly equivalent to $\calO_\Pi$. 

\smallskip

Proofs of $\Stab_\Gamma \not\in {\bf Order}_3 \cup {\bf Edge}_3$ and $P_G \not\in {\bf Order}_3 \cup {\bf Stab}_3$ 
can be performed in the similar way to the above discussions.

\end{document}